\documentclass[10pt,final]{article}
\usepackage[body={6.5in,9in},top=1in,bottom=1in,right=1in,left=1in,dvips]{geometry}
\usepackage{lipsum}
\usepackage{amsfonts}
\usepackage{graphicx}
\usepackage{epstopdf}
\usepackage{algpseudocodex}
\usepackage{graphicx}
\usepackage{color,latexsym}
\usepackage{verbatim}
\usepackage{lscape}
\usepackage{amsmath}
\usepackage{amssymb}
\usepackage{esint}
\usepackage{multirow}
\usepackage{hhline}
\usepackage{tabularx}
\usepackage{mathtools}
\usepackage{float}
\usepackage{setspace}
\usepackage{amsthm}
\usepackage{slashbox}

\ifpdf
  \DeclareGraphicsExtensions{.eps,.pdf,.png,.jpg}
\else
  \DeclareGraphicsExtensions{.eps}
\fi

\newtheorem{prop}{Proposition}
\newtheorem{remark}{Remark}

\makeatother
\begin{document}

\title{Adaptivity in Local Kernel Based Methods for Approximating the Action of Linear Operators\thanks{This report was prepared as an account of work sponsored by an agency of the United States Government. Neither the United States Government nor any agency thereof, nor any of their employees, make any warranty, express or implied, or assumes any legal liability or responsibility for the accuracy, completeness, or usefulness of any information, apparatus, product, or process disclosed, or represents that its use would not infringe privately owned rights. Reference herein to any specific commercial product, process, or service by trade name, trademark, manufacturer, or otherwise does not necessarily constitute or imply its endorsement, recommendation, or favoring by the United States Government or any agency thereof. The views and opinions of authors expressed herein do not necessarily state or reflect those of the United States Government or any agency thereof.}
}

\author{Jonah A. Reeger %
\thanks{\textit{Email}: jonah.reeger@afit.edu%
} \\
 Department of Mathematics and Statistics\\
 US Air Force Institute of Technology, 2950 Hobson Way\\
 Wright-Patterson AFB, OH 45433, USA\\
}

\maketitle

\begin{abstract}
Building on the successes of local kernel methods for approximating the solutions to partial differential equations (PDE) and the evaluation of definite integrals (quadrature/cubature), a local estimate of the error in such approximations is developed. This estimate is useful for determining locations in the solution domain where increased node density (equivalently, reduction in the spacing between nodes) can decrease the error in the solution.  An adaptive procedure for adding nodes to the domain for both the approximation of derivatives and the approximate evaluation of definite integrals is described.  This method efficiently computes the error estimate at a set of prescribed points and adds new nodes for approximation where the error is too large.  Computational experiments demonstrate close agreement between the error estimate and actual absolute error in the approximation.  Such methods are necessary or desirable when approximating solutions to PDE (or in the case of quadrature/cubature), where the initial data and subsequent solution (or integrand) exhibit localized features that require significant refinement to resolve and where uniform increases in the density of nodes accross the entire computational domain is not possible or too burdensome.
\end{abstract}

\section{Introduction}
This article concerns the development of an efficiently computable error estimate, and some basic implementations of methods based on the estimate, for adaptively approximating the action of linear operators on a set of sufficiently smooth functions.  Kernel methods are employed, whereby a basis that includes shifts of a chosen conditionally-positive definite function--the kernel--and supplemental polynomial terms is used in the approximation. The ideas behind kernel methods allow for efficient function approximation with high orders of accuracy even in the presence of variable node spacing \cite{BFNF2015b}.  Such methods are also easily generalized to high dimensions and adaptable to a wide variety of domains.

The work here expands on the successes of kernel methods in interpolation, the approximation of solutions to PDEs and the approximate evaluation of definite integrals \cite{JARBFMLW2016,JARBF2017,JAR2020,JAR2022,JGJAR2022,Hardy,Kansa,ASMV06,FHNWW14,JARBF2016} to include ``$h$-adaptivity"--increasing or decreasing local node density to improve accuracy or efficiency.  Adaptivity is necessary or desirable when the initial data and subsequent solution to a PDE exhibits localized features that require increased node density to resolve, and where uniform refinement across the entire domain is not possible or too burdensome.  This is likewise true in quadrature/cubature when the integrand has similarly localized features.

Existing investigations of adaptive kernel methods focus mainly on selection of the shape parameter in positive-definite kernels, which balances accuracy and numerical stability, or mesh refinement based on global Radial Basis Function (RBF) approximations (see, e.g., \cite{GuJung2020,GU2021113036,SARRA200579,ARBF2020,article2,728359}). Nearly all of these focus strictly on the problem of interpolation.  Still, in the context of solving PDE with local kernel methods there has recently been consideration of an $h$-adaptive RBF generated finite differences (RBF-FD) like method for the Poisson equation \cite{9854342}.  The use of local approximations is intended to reduce overall computational complexity and promote numerical stability.

Kernel approximations are advantageous because in the settings of interpolation and differentiation a point cloud is all that is necessary to describe the geometry of the domain; that is, the approximations are truly meshless.  This is in contrast to other classical methods that require a mesh of the domain, often with restrictions on the quality of the mesh.  The absence of requirements to have a mesh of a certain quality opens up opportunities for new, more flexible node refinement strategies--a benefit for $h$-adaptivity.  A recent work considered adaptive numerical differentiation with a basis including only multivariate polynomial terms, but was fixed to discrete Leja points when constructing the interpolant \cite{AdaptivePolynomial}.

The following section \ref{sec:Interpolation} formulates the problem of local kernel approximations and introduces the interpolants that are used in approximation.  Then, section \ref{sec:Error_Estimation} describes the error in approximation using kernel interpolants and an estimate of that error.  This is followed by an introduction to an adaptive algorithm that successfully uses the error estimate to achieve a prescribed tolerance in the local approximation of linear operators in section \ref{sec:Implementation}.  Finally, sections \ref{sec:Experiments} and \ref{sec:Conclusions} provide computational experiments and conclusions, respectively.

\section{Local Kernel Approximations Via Interpolation} \label{sec:Interpolation}

Consider the evaluation of $\mathcal{L}f$, where $\mathcal{L}$ is a linear operator and $f:\mathbb{R}^{d}\mapsto\mathbb{R}$.  The domain of $\mathcal{L}$ is $C^\infty(\Omega)$, where $\Omega\subset\mathbb{R}^{d}$, and the codomain of $\mathcal{L}$ could be, for example, $C^{\infty}(\Omega)$ or simply $\mathbb{R}$.  Suppose that $\mathcal{S}=\left\{\mathbf{x}_{i}\right\}_{i=1}^{N}$ is a set of $N$ unique points in $\mathbb{R}^{d}$. Further define a set of points $\mathcal{X}_{0}=\left\{\mathbf{x}_{k,0}\right\}_{k=1}^{K}$ and associate to each of these points a set $\mathcal{N}_{k,n}=\left\{\mathbf{x}_{k,j}\right\}_{j=1}^{n}$ of the $n$ points in $\mathcal{S}$ nearest to $\mathbf{x}_{k,0}$.  The choice of the points to include in $\mathcal{X}_{0}$ is dependent on the action of $\mathcal{L}$.  For instance, in the case of $\mathcal{L}$ being a derivative it is convenient to set $\mathcal{X}_{0}=\mathcal{S}_{N}$.  On the other hand, if $\mathcal{L}$ is a definite integral, then it is useful to construct on the set $\mathcal{S}$ a tessellation $T=\left\{t_{k}\right\}_{k=1}^{K}$ (via Delaunay tessellation or some other algorithm) of $K$ simplices and to let $\mathcal{X}_{0}$ be the set of the midpoints of the simplices (for each simplex, the midpoint is the average of its vertices, i.e., its barycenter).  For $d=1$ these simplices are intervals of the real line, and for $d=2$ and $d=3$ these are triangles and tetrahedra, respectively.  Throughout the remainder of this work, multiple subscripting using $k$, $n$ and $m$, in particular, will be used to explicitly indicate those variables that depend on each of these parameters.

Approximation of the action of $\mathcal{L}$ on $f$ is accomplished in a manner similar to the concept of RBF-FD where $f(\mathbf{x})$ is first approximated locally by an interpolant, and then $\mathcal{L}$ is applied to the interpolant.  For instance, when considering the approximation of a definite integral over a domain $\Omega$, it convenient to let
\begin{align}
    \mathcal{L}f = \int\limits_{\Omega}f(\mathbf{x})d\mathbf{x}=\sum\limits_{k=1}^{K}\int\limits_{\omega_{k}}f(\mathbf{x})d\mathbf{x}=\sum\limits_{k=1}^{K}\mathcal{L}_{k}f,\nonumber
\end{align}
where $\omega_{k}$ is some portion of $\Omega$ associated with $t_{k}$ and $\mathcal{L}_{k}f$ is to be understood as the integral of $f$ over $\omega_{k}$.  There is an assumption here that $\bigcup_{k=1}^{k}\omega_{k}$ and that the intersection of $\omega_{k}$ and $\omega_{k'}$ is at most a facet shared by the simplices when $k\neq k'$ so that the integral can be decomposed as the sum above.  On the other hand, when $\mathcal{L}$ is a derivative operator, $\mathcal{L}_{k}$ is considered to be the derivative at the point $\mathbf{x}_{k,0}$.  In both cases of integration and differentiation, $\mathcal{L}_{k}$ is applied to a local interpolant of $f$ as the approximation of the action of $\mathcal{L}_{k}$ on $f$.  The following two subsections describe two equivalent approaches to constructing the local interpolant.

\subsection{Approximation on the Basis of Shifts of the Kernel and Polynomials}

Consider a linear combination of (conditionally-) positive definite kernels, $\varphi$, evaluated at a set of center points,
\begin{align}
\phi_{k,n,j}(\mathbf{x}):=\varphi\left(\left\lVert \mathbf{x}-\mathbf{x}_{k,j}\right\rVert_{2}\right), j=1,2,\ldots,n\nonumber
\end{align}
and (supplemental) multivariate polynomial (multinomial) terms.

Define $\{\pi_{k,l}(\mathbf{x})\}_{l=1}^{M_{d,m}}$, with
\begin{align}
M_{d,m}=\left(\begin{array}{c}m+d\\m\end{array}\right),\nonumber
\end{align}
to be a set of all of the multivariate polynomial terms up to total degree $m$. Using multi-index notation these terms have the form $\pi_{k,l}(\mathbf{x})=(\mathbf{x}-\mathbf{x}_{k,0})^{\boldsymbol\alpha_{l}}$, with $\boldsymbol\alpha_{l}=(\alpha_{l,1},\alpha_{l,2},\ldots,\alpha_{l,d})$ ($\alpha_{l,j}>0$, $j=1,2,\ldots,d$) a multi-index and $|\boldsymbol\alpha_{l}|\leq m$. As a reminder, with multi-indices (see, e.g., \cite{hunter2001applied} section 11.1)
\begin{itemize}
\item $\mathbf{x}^{\boldsymbol{\alpha}_{l}}=x_{1}^{\alpha_{l,1}}x_{2}^{\alpha_{l,2}}\cdots x_{d}^{\alpha_{l,d}}$,
\item $|\boldsymbol{\alpha}_{l}|=\sum_{j=1}^{d}\alpha_{l,d}$,
\item $\partial^{\boldsymbol{\alpha}_{l}}=\frac{\partial^{\alpha_{l,1}}}{\partial x_{1}^{\alpha_{l,1}}}\frac{\partial^{\alpha_{l,2}}}{\partial x_{2}^{\alpha_{l,2}}}\cdots\frac{\partial^{\alpha_{l,d}}}{\partial x_{d}^{\alpha_{l,d}}}$, and
\item $\boldsymbol\alpha_{l}!=(\alpha_{l,1}!)(\alpha_{l,2}!)\cdots(\alpha_{l,d}!)$.
\end{itemize}
The interpolant is constructed as
\begin{align}
s_{k,n,m}[f](\mathbf{x}):=\sum_{j=1}^{n}\lambda_{k,n,m,j}[f]\phi_{k,n,j}\left(\mathbf{x}\right)+\sum_{l=1}^{M_{d,m}}\gamma_{k,n,m,l}[f]\pi_{k,l}(\mathbf{x}).\nonumber
\end{align}
Likewise,
\begin{align}
s_{k,n,m}[f](\mathbf{x}):=\left[\begin{array}{c}\boldsymbol\lambda_{k,n,m}[f]\\\boldsymbol\gamma_{k,n,m}[f]\end{array}\right]^{T}\left[\begin{array}{c}\boldsymbol{\Phi}_{k,n}(\mathbf{x}) \\ \boldsymbol{\Pi}_{k,m}(\mathbf{x})\end{array}\right].\nonumber
\end{align}
To ensure that $s_{k,n,m}[f]$ interpolates $f$ at the set of points in $\mathcal{N}_{k,n}$ the coefficient vectors
\begin{align}
\boldsymbol{\lambda}_{k,n,m}[f]=\left[\begin{array}{cccc} \lambda_{k,n,m,1}[f] & \lambda_{k,n,m,2}[f] & \cdots & \lambda_{k,n,m,n}[f] \end{array}\right]^{T}\nonumber
\end{align}
and
\begin{align}
\boldsymbol{\gamma}_{k,n,m}[f]=\left[\begin{array}{cccc} \gamma_{k,n,m,1}[f] & \gamma_{k,n,m,2}[f] & \cdots & \gamma_{k,n,m,M_{d,m}}[f]\end{array}\right]^{T}\nonumber
\end{align}
are chosen to satisfy the interpolation conditions ($j=1,2,\ldots,n$),
\begin{align}
s_{k,n,m}[f](\mathbf{x}_{k,j})=f(\mathbf{x}_{k,j}).\nonumber
\end{align}
This leads to an underdetermined system of linear equations to solve for the coefficient vectors, so the typical constraints applied to kernel-based interpolants where the kernel is conditionally positive definite are also enforced, i.e., ($l=1,2,\ldots,M_{d,m}$)
\begin{align}
\sum_{j=1}^{n}\lambda_{k,n,m,j}[f]\pi_{k,l}(\mathbf{x}_{k,j})=0.\nonumber
\end{align}
The $n\times 1$ vector $\boldsymbol{\Phi}_{k,n}(\mathbf{x})$ and $M_{d,m}\times1$ vector $\boldsymbol{\Pi}_{k,m}(\mathbf{x})$ consist of all of the basis functions evaluated at $\mathbf{x}$, i.e.,
\begin{align}
   \boldsymbol{\Phi}_{k,n}(\mathbf{x})=\left[\begin{array}{cccc} \phi_{k,n,1}\left(\mathbf{x}\right) & \phi_{k,n,2}\left(\mathbf{x}\right)& \cdots & \phi_{k,n,n}\left(\mathbf{x}\right)\end{array}\right]^{T}\nonumber
\end{align}
and
\begin{align}
   \boldsymbol{\Pi}_{k,m}(\mathbf{x})=\left[\begin{array}{cccc} \pi_{k,1}\left(\mathbf{x}\right) & \pi_{k,2}\left(\mathbf{x}\right)& \cdots & \pi_{k,n}\left(\mathbf{x}\right)\end{array}\right]^{T}.\nonumber
\end{align}

Satisfaction of the interpolation conditions and constraints amounts to solving the system of linear equations
\begin{align}
S_{k,n,m}\left[\begin{array}{c}\boldsymbol{\lambda}_{k,n,m}[f] \\ \boldsymbol{\gamma}_{k,n,m}[f]\end{array}\right]=\left[\begin{array}{c}\mathbf{f}_{k,n} \\ \mathbf{0}_{M_{d,m}}\end{array}\right]\nonumber
\end{align}
with the $(n+M_{d,m})\times(n+M_{d,m})$ matrix
\begin{align}
S_{k,n,m}=\left[\begin{array}{cc}\Phi_{k,n}^{T} & P_{k,n,m}\\P_{k,n,m}^{T} & 0_{M_{d,m}\times M_{d,m}}\end{array}\right]\nonumber
\end{align}
and the right hand side consisting of the function $f$ evaluated at the points in $\mathcal{N}_{k,n}$,
\begin{align}
\mathbf{f}_{k,n}=\left[\begin{array}{cccc}f(\mathbf{x}_{k,1}) & f(\mathbf{x}_{k,2}) & \cdots & f(\mathbf{x}_{k,n}) \end{array}\right]^T,\nonumber
\end{align}
and an $M_{d,m}\times 1$ vector of zeros, $\mathbf{0}_{M_{d,m}}$.  The $n\times n$ submatrix $\Phi_{k,n}$ is made up of the kernel basis elements evaluated at each point in $\mathcal{N}_{k,n}$, that is
\begin{align}
[\Phi_{k,n}]_{ij}=\phi_{k,n,j}\left(\mathbf{x}_{k,i}\right),\mbox{ for }i,j=1,2,\ldots,n.\nonumber
\end{align}
Likewise the $n\times M_{d,m}$ Vandermonde matrix $P_{k,n,m}$ consists of the polynomial basis evaluated at each point in $\mathcal{N}_{k,n}$ so that
\begin{align}
[P_{k,n,m}]_{il}=\pi_{k,l}(\mathbf{x}_{k,i}),\mbox{ for }i=1,2,\ldots,n\mbox{ and }l=1,2,\ldots,M_{d,m}.\nonumber
\end{align}
Assuming that the kernel $\varphi$ is conditionally positive-definite of order $m$ and the set $\mathcal{N}_{k,n}$ is unisolvent on the space, $\mathbb{P}_{m}^{d}$, of $d$-variate polynomials up to degree $m$, the matrix $S_{k,n,m}$ is invertible and the interpolation problem has a unique solution \cite{HW2005}.

\subsection{The Lagrange Form of the Interpolant}

An alternative formulation of the interpolant, more convenient for the presentation of theoretical results, is written
\begin{align}
 s_{k,n,m}[f](\mathbf{x})= \sum\limits_{i=1}^{n}\psi_{k,n,m,i}(\mathbf{x})f(\mathbf{x}_{k,i}),\nonumber
\end{align}
where the new set of basis functions satisfy the cardinal property
\begin{align}
    \psi_{k,n,m,i}(\mathbf{x}_{k,j})=\left\{\begin{array}{cc} 1 & i=j \\ 0 & i\neq j\end{array}\right..\nonumber
\end{align}
The two sets of basis functions are related by
\begin{align}
    S_{k,n,m}\left[\begin{array}{c}\boldsymbol\Psi_{k,n,m}(\mathbf{x})\\\boldsymbol\Xi_{k,n,m}(\mathbf{x})\end{array}\right]=\left[\begin{array}{c}\boldsymbol\phi_{k,n}(\mathbf{x})\\\boldsymbol{\pi}_{k,m}(\mathbf{x})\end{array}\right]\label{eq:cardinalrelationship}
\end{align}
with only the $n\times 1$ vector
\begin{align}
    \boldsymbol\Psi_{k,n,m}(\mathbf{x}) = \left[\begin{array}{cccc}\psi_{k,n,m,1}(\mathbf{x}) & \psi_{k,n,m,2}(\mathbf{x}) & \cdots & \psi_{k,n,m,n}(\mathbf{x})\end{array}\right].\nonumber
\end{align}
required to form $s_{k,n,m}[f]$.

\section{Error Estimation in Local Kernel Approximations} \label{sec:Error_Estimation}

An estimate of the error in the approximate application of $\mathcal{L}$ is necessary to determine locations in the domain of interest that require refinement to achieve a desired tolerance.  In the following subsections a convenient expression of the pointwise error and an estimate that closely matches that error are discussed.  Further, an expression for the interpolation coefficients that correspond to the shifts of the kernel in the basis used for approximation (that is, the coefficient vector $\boldsymbol{\lambda}_{k,n,m}[f]$) is given in terms of linear combinations of the weights of strictly polynomial based approximations of a set of (mixed partial) derivatives.

\subsection{Pointwise Error in the Local Kernel Based Interpolant}

For reasons that will be made clear in the following sections let $\mu\in\mathbb{Z}$ and $\mu\geq1$.  The Taylor formula of a function $f(\mathbf{x})$ about the point $\mathbf{x}_{k,0}$, with $f$ having continuous mixed partial derivatives up to order $m+\mu+1$ in a neighborhood of $\mathbf{x}_{k,0}$, can be written as \cite{MultipointTaylor}
\begin{align}
    f(\mathbf{x}) =  \sum_{l=1}^{M_{d,m}}\frac{1}{\boldsymbol\alpha_{l} !} \partial^{\boldsymbol\alpha_{l}}f(\mathbf{x})\big|_{\mathbf{x}=\mathbf{x}_{k,0}}\pi_{k,l}(\mathbf{x})+\sum_{l=M_{d,m}+1}^{M_{d,m+\mu}}\frac{1}{\boldsymbol\alpha_{l} !} \partial^{\boldsymbol\alpha_{l}}f(\mathbf{x})\big|_{\mathbf{x}=\mathbf{x}_{k,0}}\pi_{k,l}(\mathbf{x})+R_{m+\mu}[f](\mathbf{x}) \label{eq:expanded_taylor}
\end{align}
where the remainder term is expressible as
\begin{align}
R_{m+\mu}[f](\mathbf{x})=\sum\limits_{l=M_{d,m+\mu}+1}^{M_{d,m+\mu+1}}\frac{m+\mu+1}{\boldsymbol\alpha_{l}!}\int\limits_{0}^{1}\frac{\partial^{\boldsymbol\alpha_{l}}f(\mathbf{y})|_{\mathbf{y}=\mathbf{x}+t(\mathbf{x}-\mathbf{x}_{k,0})}}{(m+\mu)!}(1-t)^{m+\mu}dt(\mathbf{x}-\mathbf{x}_{k,0})^{\boldsymbol\alpha_{l}}.\nonumber
\end{align}
When evaluated inside the convex hull of $\mathcal{N}_{k,n}$ this remainder term behaves as $O(h_{k,n}^{m+\mu+1})$ as $h_{k,n}\to0$, where $h_{k,n}$ is the typical spacing between the points in $\mathcal{N}_{k,n}$.

\begin{prop}
Suppose the $f$ has continuous mixed partial derivatives up to order $m+\mu+1$ in a neighborhood of $\mathbf{x}_{k,0}$.  Further assume that the kernel $\varphi$ is conditionally positive-definite of order $m$ and the set $\mathcal{N}_{k}$ is unisolvent on the space, $\mathbb{P}_{m}^{d}$, of $d$-variate polynomials up to degree $m$.  The point-wise error in the kernel based interpolant $s_{k,n,m}[f]$ is
\begin{align}
s_{k,n,m}&[f](\mathbf{x})-f(\mathbf{x})\nonumber\\=&\sum\limits_{l=1}^{M_{d,m+\mu}-M_{d,m}}\frac{1}{\boldsymbol\alpha_{M_{d,m}+l}!}\partial^{\boldsymbol\alpha_{M_{d,m}+l}}f(\mathbf{x})\big|_{\mathbf{x}=\mathbf{x}_{k,0}} E_{k,n,m,M_{d,m}+l}(\mathbf{x})+\left(s_{k,n,m}[R_{m+\mu}[f]](\mathbf{x})-R_{m+\mu}[f](\mathbf{x})\right). \label{eq:interperrorsums}
\end{align}
with
\begin{align}
E_{k,n,m,M_{d,m}+l}(\mathbf{x})=s_{k,n,m}[\pi_{k,m,M_{d,m}+l}](\mathbf{x})-\pi_{k,m,M_{d,m}+l}(\mathbf{x})\nonumber
\end{align}
the error in approximating the polynomial term $\pi_{k,m,M_{d,m}+l}$ with a kernel interpolant.
\end{prop}
\begin{proof}
Existence of the unique interpolant $s_{k,n,m}[f]$ follows from the kernel $\varphi$ being conditionally positive-definite of order $m$ and the set $\mathcal{N}_{k,n}$ being unisolvent on the space, $\mathbb{P}_{m}^{d}$, of $d$-variate polynomials up to degree $m$.  To develop a convenient expression for the error in the kernel based interpolant, it is useful to evaluate the Taylor formula at each point in $\mathcal{N}_{k,n}$ and write
\begin{align}
    s_{k,n,m}[f]=&\sum_{i=1}^{n}\psi_{k,n,m,i}(\mathbf{x})\left(\sum_{l=1}^{M_{d,m}}\frac{1}{\boldsymbol\alpha_{l} !} \partial^{\boldsymbol\alpha_{l}}f(\mathbf{x})\big|_{\mathbf{x}=\mathbf{x}_{k,0}}\pi_{k,l}(\mathbf{x}_{k,i})\right)+\cdots\nonumber\\
    &\sum_{i=1}^{n}\psi_{k,n,m,i}(\mathbf{x})\left(\sum_{l=M_{d,m}+1}^{M_{d,m+\mu}}\frac{1}{\boldsymbol\alpha_{l} !} \partial^{\boldsymbol\alpha_{l}}f(\mathbf{x})\big|_{\mathbf{x}=\mathbf{x}_{k,0}}\pi_{k,l}(\mathbf{x}_{k,i})\right)+\cdots\nonumber\\
    &\sum_{i=1}^{n}\psi_{k,n,m,i}(\mathbf{x})R_{m+\mu}[f](\mathbf{x}_{k,i}).\nonumber
\end{align}
Swapping the orders of summation in the first two lines of the expression reveals
\begin{align}
    s_{k,n,m}[f]=&\sum_{l=1}^{M_{d,m}}\frac{1}{\boldsymbol\alpha_{l} !} \partial^{\boldsymbol\alpha_{l}}f(\mathbf{x})\big|_{\mathbf{x}=\mathbf{x}_{k,0}}\left(s_{k,n,m}[\pi_{k,l}](\mathbf{x})\right)+\cdots\nonumber\\
    &\sum_{l=M_{d,m}+1}^{M_{d,m+\mu}}\frac{1}{\boldsymbol\alpha_{l} !} \partial^{\boldsymbol\alpha_{l}}f(\mathbf{x})\big|_{\mathbf{x}=\mathbf{x}_{k,0}}\left(s_{k,n,m}[\pi_{k,l}](\mathbf{x})\right)+\cdots\nonumber\\
    &\left(s_{k,n,m}[R_{m+\mu}[f]](\mathbf{x}) \right).\label{eq:expanded_interpolant}
\end{align}
Notice that $s_{k,n,m}[\pi_{k,l}](\mathbf{x})=\pi_{k,l}(\mathbf{x})$ for $l\leq M_{d,m}$ so that the first sum in \eqref{eq:expanded_interpolant} is identical to that in \eqref{eq:expanded_taylor}.  Therefore
\begin{align}
    s_{k,n,m}&[f](\mathbf{x})-f(\mathbf{x})\nonumber\\=&\sum\limits_{l=M_{d,m}}^{M_{d,m+\mu}}\frac{1}{\boldsymbol\alpha_{l}!}\partial^{\boldsymbol\alpha_{l}}f(\mathbf{x})\big|_{\mathbf{x}=\mathbf{x}_{k,0}} \left(s_{k,n,m}[\pi_{k,m,l}](\mathbf{x})-\pi_{k,m,l}(\mathbf{x})\right)+\left(s_{k,n,m}[R_{m+\mu}[f]](\mathbf{x})-R_{m+\mu}[f](\mathbf{x})\right).\nonumber
\end{align}
Utilizing the definition of $E_{k,n,m,M_{d,m}+l}$ and changing the indexing on the sum produces the desired result.
\end{proof}

\begin{remark} \label{rem:Approx_Order}
Although in this work $\mu>0$, if instead $\mu=0$, then the first term in \eqref{eq:interperrorsums} does not appear, and this error formula is analogous to the one presented in \cite{VB2019} but in a different form.  Utilizing similar arguments, the first sum in \eqref{eq:interperrorsums} behaves as $O(h_{k,n}^{m+1})$ as $h_{k,n}\to0$, while the second term behaves as $O(h_{k,n}^{m+\mu+1})$ as $h_{k,n}\to0$.  Therefore, the error in the interpolant is dominated by the first term.
\end{remark}

\subsection{Finite Difference Expressions of Kernel Interpolation Coefficients} \label{sec:finitediffcoeffs}

As long as $n>M_{d,m}$, the constraints $\sum_{j=1}^{n}\lambda_{k,m,j}[f]\pi_{k,l}(\mathbf{x}_{k,j})=0$, $l=1,2,\ldots,M_{d,m}$, form an underdetermined system of equations $P_{k,n,m}^{T}\boldsymbol{\lambda}_{k,n,m}=\mathbf{0}_{M_{d,m}}$.  If the set $\mathcal{N}_{k,n}$ is unisolvent with respect to the span of the set $\{\pi_{k,l}(\mathbf{x})\}_{l=1}^{M_{d,m}}$, the dimension of the null space of $P_{k,n,m}^{T}$ is $n-M_{d,m}$.  Suppose that $\mathbf{d}_{k,n,l}$ is the solution to
\begin{align}
    \left[\begin{array}{cccc}\pi_{k,1}(\mathbf{x}_{k,1}) & \pi_{k,1}(\mathbf{x}_{k,2}) & \cdots & \pi_{k,1}(\mathbf{x}_{k,n}) \\
    \pi_{k,2}(\mathbf{x}_{k,1}) & \pi_{k,2}(\mathbf{x}_{k,2}) & \cdots & \pi_{k,2}(\mathbf{x}_{k,n})\\
    \vdots & \vdots & & \vdots \\
    \pi_{k,l-1}(\mathbf{x}_{k,1}) & \pi_{k,l-1}(\mathbf{x}_{k,2}) & \cdots & \pi_{k,l-1}(\mathbf{x}_{k,n}) \\
    \pi_{k,l}(\mathbf{x}_{k,1}) & \pi_{k,l}(\mathbf{x}_{k,2}) & \cdots & \pi_{k,l}(\mathbf{x}_{k,n}) \\
    \pi_{k,l+1}(\mathbf{x}_{k,1}) & \pi_{k,l+1}(\mathbf{x}_{k,2}) & \cdots & \pi_{k,l+1}(\mathbf{x}_{k,n}) \\
    \vdots & \vdots & & \vdots \\
    \pi_{k,n}(\mathbf{x}_{k,1}) & \pi_{k,n}(\mathbf{x}_{k,2}) & \cdots & \pi_{k,n}(\mathbf{x}_{k,n})
    \end{array}\right]\mathbf{d}_{k,n,l}=\left[\begin{array}{c}0\\0\\\vdots\\0\\\boldsymbol{\alpha}_{l}!\\0\\\vdots\\0\end{array}\right].\nonumber
\end{align}
The set $\{\mathbf{d}_{k,n,l}\}_{l=M_{d,m}+1}^{n}$ is linearly independent and contained in the null space of $P_{k,n,m}^{T}$.  Therefore, there exists a set of weights $\beta_{k,n,m,l}[f]$, $l=M_{d,m}+1,\ldots,n$, such that
\begin{align}
    \boldsymbol{\lambda}_{k,n,m}[f]=\sum\limits_{l=M_{d,m}+1}^{n}\beta_{k,n,m,l}[f]\mathbf{d}_{k,n,l}.\nonumber
\end{align}
Supposing that $\partial^{\boldsymbol\alpha_{j}}f(\mathbf{x})$ is Lipschitz continuous in a (convex) neighborhood of $\mathbf{x}_{k,0}$, for $j=1,2,\ldots,n$, \cite{ScatteredNumericalDifferentiation}
\begin{align}
\mathbf{d}_{k,n,l}^{T}\mathbf{f}_{k,n,m}=\partial^{\boldsymbol{\alpha}_{l}}q_{k,n}[f](\mathbf{x})\big|_{\mathbf{x}=\mathbf{x}_{k,0}}\nonumber
\end{align}
is an approximation of $\partial^{\boldsymbol{\alpha}_{l}}f(\mathbf{x})\big|_{\mathbf{x}=\mathbf{x}_{k,0}}$ where
\begin{align}
    q_{k,n}[f](\mathbf{x})=\sum\limits_{l=1}^{n}\rho_{k,n,l}[f]\pi_{k,n,l}(\mathbf{x})\nonumber
\end{align}
is a polynomial interpolant satisfying $q_{k,n}[f](\mathbf{x}_{k,j})=f(\mathbf{x}_{k,j})$, $j=1,2,\ldots,n$. In what follows it is more convenient to write $\boldsymbol{\lambda}_{k,n,m}[f]=D_{k,n,m}\boldsymbol{\beta}_{k,n,m}[f]$, where
\begin{align}
D_{k,n,m+\mu}=\left[\begin{array}{cccc}\mathbf{d}_{k,n,M_{d,m}+1} & \mathbf{d}_{k,n,M_{d,m}+2} & \ldots & \mathbf{d}_{k,n,n}\end{array}\right]\nonumber
\end{align}
and
\begin{align}
\boldsymbol{\beta}_{k,n,m}[f]=\left[\begin{array}{cccc}\beta_{k,n,m,M_{d,m}+1}[f] & \beta_{k,n,m,M_{d,m}+2}[f] & \ldots & \beta_{k,n,m,n}[f]\end{array}\right]^{T}.\nonumber
\end{align}

\subsection{An Estimate of the Error in Local Kernel Based Approximations}

To construct a method that locally adapts the spacing of the points to reduce the error,
\begin{align}
\lVert\mathcal{L}_{k}s_{k,n,m}[f]-\mathcal{L}_{k}f\rVert_{2}=\lVert\mathcal{L}_{k}(s_{k,n,m}[f]-f)\rVert_{2},\nonumber
\end{align}
in the approximation of $\mathcal{L}_{k}f$ an estimate of the error must be utilized.  Here
 \begin{align}
\lVert\mathcal{L}_{k}s_{k,n,m}[f]-\mathcal{L}_{k}s_{k,n,m+\mu}[f]\rVert_{2}=\lVert\mathcal{L}_{k}(s_{k,n,m}[f]-s_{k,n,m+\mu}[f])\rVert_{2},\nonumber
\end{align}
$\mu\in\mathbb{Z}$ and $\mu\geq1$, is used as the estimate of the error.

\begin{prop}
Suppose that the kernel $\varphi$ is conditionally positive-definite of order $m$ and the set $\mathcal{N}_{k,n}$ is unisolvent on the space, $\mathbb{P}_{m+\mu}^{d}$, of $d$-variate polynomials up to degree $m+\mu$. Further, let $\partial^{\boldsymbol\alpha_{j}}f(\mathbf{x})$ be Lipschitz continuous in a convex neighborhood of $\mathbf{x}_{k,0}$ for $j=1,2,\ldots,n$.  If $n=M_{d,m+\mu}$, then
\begin{align}
s_{k,n,m}&[f](\mathbf{x})-s_{k,n,m+\mu}[f](\mathbf{x})=\nonumber\\
&\sum\limits_{l=1}^{M_{d,m+\mu}-M_{d,m}}\frac{1}{\boldsymbol\alpha_{M_{d,m}+l}!}\partial^{\boldsymbol\alpha_{M_{d,m}+l}}q_{k,n}[f](\mathbf{x})|_{\mathbf{x}=\mathbf{x}_{k,0}}E_{k,n,m,M_{d,m}+l}(\mathbf{x}).\label{eq:errorestimatesums}
\end{align}
\end{prop}
\begin{proof}
Let $n\geq M_{d,m+\mu}$ and notice that
\begin{align}
S_{k,n,m+\mu} = \left[\begin{array}{cc}S_{k,n,m} & \left[\begin{array}{c}\tilde{P}_{k,n,m+\mu} \\ 0_{(n+M_{d,m})\times(M_{d,m+\mu}-M_{d,m})}\end{array}\right] \\
\left[\begin{array}{cc}\tilde{P}_{k,n,m+\mu}^{T} & 0_{(M_{d,m+\mu}-M_{d,m})\times(n+M_{d,m})}\end{array}\right] & 0_{M_{d,m+\mu}-M_{d,m}}\end{array}\right],\nonumber
\end{align}
so that if $S_{k,n,m}$ is invertible and $P_{k,n,m+\mu}$ is full rank (a consequence of the unisolvency of $\mathcal{N}_{k,n}$), block matrix inversion of $S_{k,n,m+\mu}$ yields
\begin{align}
\Psi_{k,n,m+\mu}(\mathbf{x})=\Psi_{k,n,m}(\mathbf{x})-\Lambda_{k,n,m}\left(\tilde{P}_{k,n,m+\mu}^{T}\Lambda_{k,n,m}\right)^{-1}\left(\tilde{P}_{k,n,m}^{T}\Psi_{k,n,m}(\mathbf{x})-\tilde{\boldsymbol\Pi}_{k,n,m+\mu}(\mathbf{x})\right),\nonumber
\end{align}
where
\begin{align}
\Lambda_{k,n,m}=\left[\begin{array}{cccc}\boldsymbol{\lambda}_{k,n,m}[\pi_{k,M_{d,m}+1}] & \boldsymbol{\lambda}_{k,n,m}[\pi_{k,M_{d,m}+2}] & \cdots & \boldsymbol{\lambda}_{k,n,m}[\pi_{k,M_{d,m+\mu}}] \end{array}\right]\nonumber
\end{align}
is the matrix with columns consisting of the interpolation coefficients corresponding to the kernel basis elements when interpolating polynomial the polynomial basis elements of degree $m+1,m+2,\ldots,m+\mu$ using only the kernel basis set supplemented by polynomial terms up to degree $m$. The results of section \ref{sec:finitediffcoeffs} suggest that the matrix $\Lambda_{k,n,m}$ can be written as
\begin{align}
    \Lambda_{k,n,m}=D_{k,n,m+\mu}B_{k,n,m},\nonumber
\end{align}
where
\begin{align}
B_{k,n,m}=\left[\begin{array}{cccc}\boldsymbol{\beta}_{k,n,m}[\pi_{k,M_{d,m}+1}] & \boldsymbol{\beta}_{k,n,m}[\pi_{k,M_{d,m}+2}] & \cdots & \boldsymbol{\beta}_{k,n,m}[\pi_{k,M_{d,m+\mu}}] \end{array}\right].\nonumber
\end{align}
Further,
\begin{align}
\tilde{P}_{k,n,m+\mu}^{T}\Lambda_{k,n,m} =&(D_{k,n,m+\mu}^{T}\tilde{P}_{k,n,m+\mu})^{T}\hat{B}_{k,n,m}\nonumber\\
=&A_{m+\mu}\hat{B}_{k,n,m}\nonumber
\end{align}
with
\begin{align}
A_{m+\mu}= \left[\begin{array}{cccc}\boldsymbol{\alpha}_{M_{d,m}+1}! & & & \\
& \boldsymbol{\alpha}_{M_{d,m}+2}! & & \\
& & \ddots & \\
& & & \boldsymbol{\alpha}_{M_{d,m+\mu}}!\end{array}\right].\nonumber
\end{align}
and $\hat{B}_{k,n,m}$ the first $M_{d,m+\mu}-M_{d,m}$ rows of $B_{k,n,m}$.  Therefore, after substitution
\begin{align}
\Psi_{k,n,m+\mu}(\mathbf{x})=\Psi_{k,n,m}(\mathbf{x})-D_{k,n,m+\mu}B_{k,n,m}\left(\hat{B}_{k,n,m}\right)^{-1}A_{m+\mu}^{-1}\left(\tilde{P}_{k,n,m}^{T}\Psi_{k,n,m}(\mathbf{x})-\tilde{\boldsymbol\Pi}_{k,n,m+\mu}(\mathbf{x})\right).\nonumber
\end{align}
Notice that
\begin{align}
s_{k,n,m}&[f](x)-s_{k,n,m+\mu}[f](\mathbf{x}) \nonumber\\
=& (\Psi_{k,n,m}(\mathbf{x})-\Psi_{k,n,m+\mu}(\mathbf{x}))^{T}\mathbf{f}_{k,n}\nonumber\\
=&\left(D_{k,n,m+\mu}B_{k,n,m}\left(\hat{B}_{k,n,m}\right)^{-1}A_{m+\mu}^{-1}\left(\tilde{P}_{k,n,m}^{T}\Psi_{k,n,m}(\mathbf{x})-\tilde{\boldsymbol\Pi}_{k,n,m+\mu}(\mathbf{x})\right)\right)^{T}\mathbf{f}_{k,n},\nonumber
\end{align}
so that using the definition of $E_{k,n,m,M_{d,m}+l}$
\begin{align}
s_{k,n,m}&[f](x)-s_{k,n,m+\mu}[f](\mathbf{x}) =\nonumber\\ &\sum\limits_{l=1}^{M_{d,m+\mu}-M_{d,m}}\frac{1}{\boldsymbol\alpha_{M_{d,m}+l}!}\mathbf{d}_{k,n,m,M_{d,m}+l}^{T}\mathbf{f}_{k,n}E_{k,n,m,M_{d,m}+l}(\mathbf{x})+\cdots\nonumber\\
&\sum\limits_{i=1}^{n-M_{d,m+\mu}}\mathbf{d}_{k,n,m,M_{d,m+\mu}+i}^{T}\mathbf{f}_{k,n}\sum\limits_{l=1}^{M_{d,m+\mu}-M_{d,m}}C_{k,n,m,li}\frac{1}{\boldsymbol\alpha_{M_{d,m}+l}!}E_{k,n,m,M_{d,m}+l}(\mathbf{x})\label{eq:full_estimate}
\end{align}
with $C_{k,n,m}=\tilde{B}_{k,n,m}\left(\hat{B}_{k,n,m}\right)^{-1}$.  Taking $n=M_{d,m+\mu}$ the second sum on the right does not appear in the expression, and utilizing the continuity assumption on the derivatives of $f$ leads to the desired result.
\end{proof}

\begin{remark} \label{rem:estimate1}
    The error estimate \eqref{eq:errorestimatesums} is an approximation to the dominant term in the expression \eqref{eq:interperrorsums}.  According to \cite{ScatteredNumericalDifferentiation}
     \begin{align}
     \left\lvert \partial^{\boldsymbol\alpha_{l}}f(\mathbf{x})|_{\mathbf{x}=\mathbf{x}_{k,0}}-\partial^{\boldsymbol\alpha_{l}}q_{k,n}[f](\mathbf{x})|_{\mathbf{x}=\mathbf{x}_{k,0}}\right\rvert\leq O(h_{k,n}^{m+\mu+1-|\boldsymbol\alpha_{l}|})\mbox{ as }h_{k,n}\to0,\nonumber
     \end{align}
     with
        \begin{align}
    h_{k,n}=\max\limits_{j=1,2,\ldots,n}\left\lVert \mathbf{x}_{k,j}-\mathbf{x}_{k,0}\right\rVert_{2}.\nonumber
    \end{align}
    The terms \cite{VB2019}
     \begin{align}
     s_{k,n,m}[\pi_{k,m,l}](\mathbf{x})-\pi_{k,m,l}(\mathbf{x})\nonumber
     \end{align}
     are also at most $O(h_{k,n}^{|\boldsymbol\alpha_{l}|})$.  Therefore, the difference between the error estimate and the dominant term in the error in the kernel based interpolant that includes polynomial terms up to degree $m$ is at most $O(h^{m+\mu+1})$ as $h_{k,n}\to0$.  For any choice of $\mu>0$, this is at least one order higher than the size of the dominant term in the error in the kernel interpolant.  That is, this is estimate provides a reasonable approximation of the error.
\end{remark}

\begin{remark}
    For $n>M_{d,m+\mu}$ the second sum on the right hand side of \eqref{eq:full_estimate} does not readily appear to be an approximation for the second term on the right hand side of \eqref{eq:interperrorsums}.  However, the discussion in remark \ref{rem:estimate1} applies similarly and these terms are also at least one order higher than the size of the dominant term in the error in the kernel interpolant. Unless otherwise stated the computations presented herein use $n=M_{d,m+\mu}$.
\end{remark}

\section{An Implementation of Adaptive Local Kernel Approximations} \label{sec:Implementation}

Given the error estimate developed in the previous sections, an initial implementation of an $h$-adaptive method for approximating action of a linear operator is now presented.  This implementation only adds nodes to the set $\mathcal{S}$ to decrease the spacing between nodes locally (refinement), but it does not remove nodes where the error estimate indicates the error is smaller than the prescribed tolerance (derefinement).  Efficient computation of the estimate relies on the similarities between the systems of linear equations used to solve for weights when including in the basis for approximation polynomials up to order $m$ and polynomials up to order $m+\mu$.

\subsection{Weight Computation}

Application of $\mathcal{L}_{k}$ to the interpolant written in the cardinal basis reveals
\begin{align}
\mathcal{L}_{k}s_{k,n,m}[f] = \sum\limits_{i=1}^{n}\left(\mathcal{L}_{k}\psi_{k,n,m,i}\right)f_{k,n,i}
\end{align}
Defining
\begin{align}
\mathbf{w}_{k,n,m}=\left(\mathcal{L}_{k}\psi_{k,n,m,i}\right)\nonumber
\end{align}
the approximate operation amounts to an inner product between a vector containing values of the function to which the operation is being applied and a vector of ``weights" containing the operator $\mathcal{L}_{k}$ applied to the cardinal basis elements. This is a standard method for determining approximation weights, e.g., as in the case of finite difference or pseudospectral approximations of derivatives or for Newton-Cotes quadrature weights for approximating definite integrals.  In any case, these weights are independent of the function $f$.  However, construction of the cardinal basis to determine these weights is unnecessary, and instead an equivalent process determines the weight set $\mathbf{w}_{k,n,m}$ utilizing \eqref{eq:cardinalrelationship} as
\begin{align}
S_{k,n,m}\left[\begin{array}{c}\mathbf{w}_{k,n,m}\\\mathbf{v}_{k,n,m}\end{array}\right]=\left[\begin{array}{c}\mathcal{L}\boldsymbol{\Phi}_{k,n} \\ \mathcal{L}\boldsymbol{\Pi}_{k,m}\end{array}\right].\label{eq:weight_ls}
\end{align}
Here $\mathcal{L}_{k}\boldsymbol{\Phi}_{k,n}$ and $\mathcal{L}_{k}\boldsymbol{\Pi}_{k,m}$ are to be understood as the vectors obtained by entry-wise application of $\mathcal{L}_{k}$ to the entries of $\boldsymbol{\Phi}_{k,n}$ and $\boldsymbol\Pi_{k,m}$, respectively.

\subsection{Reduced Computational Cost for the Error Estimate} \label{sec:Computational_Cost}

Given that $S_{k,n,m}$ is the upper left block of $S_{k,n,m+\mu}$, computation of the error estimate can be completed more efficiently than solving two full systems of linear equations.  Block matrix inversion of $S_{k,n,m+\mu}$ yields that if the solution to \eqref{eq:weight_ls} is available, then
\begin{align}
    \mathbf{w}_{k,n,m+\mu}=&\mathbf{w}_{k,n,m}-\Lambda_{k,n,m}(\tilde{P}_{k,n,m+\mu}^{T}\Lambda_{k,n,m})^{-1}\left(\mathcal{L}\tilde{\boldsymbol{\Pi}}_{k,m+\mu}-\tilde{P}_{k,n,m+\mu}^{T}\mathbf{w}_{k,n,m}\right) \label{eq:efficient_solve}
\end{align}
The cost of determining $\mathbf{w}_{k,n,m}$ is $O((n+M_{d,m})^{3})$, while the dominant part of the cost of determining $\mathbf{w}_{k,n,m+\mu}$ is only a further $O(n(M_{d,M+\mu}-M_{d,m})^{2})$ incurred for the multiplication $\tilde{P}_{k,n,m+\mu}^{T}\Lambda_{k,n,m}$.  Now, to guarantee unique values of $\mathbf{w}_{k,n,m+\mu}$ it is necessary that $n\geq M_{d,m+\mu}$.  The expressions for the terms in these costs can be better understood by noting
\begin{align}
    M_{d,m+\mu}=& M_{d,m}\frac{(m+\mu+d)!}{(m+\mu)!}\frac{m!}{(m+d)!},\nonumber
\end{align}
so that when $n\approx M_{d,m+\mu}$ the cost behaves as
\begin{align}
O\left(M_{d,m}^{3}\frac{(m+\mu+d)!}{(m+\mu)!}\frac{m!}{(m+d)!}\left(\frac{(m+\mu+d)!}{(m+\mu)!}\frac{m!}{(m+d)!}-1\right)^{2}\right)\nonumber
\end{align}
and there is computational efficiency compared to solving
\begin{align}
S_{k,n,m+\mu}\mathbf{w}_{k,n,m+\mu}=\left[\begin{array}{c}\mathcal{L}\boldsymbol{\Phi}_{k,n} \\ \mathcal{L}\boldsymbol{\Pi}_{k,m+\mu}\end{array}\right].\nonumber
\end{align}
Define $\tau_{d,m}$ and $\tau_{d,m+\mu}$ be the times it takes to solve \eqref{eq:weight_ls} and
\begin{align}
S_{k,n,m+\mu}\left[\begin{array}{c}\mathbf{w}_{k,n,m+\mu}\\\mathbf{v}_{k,n,m+\mu}\end{array}\right]=\left[\begin{array}{c}\mathcal{L}\boldsymbol{\Phi}_{k,n} \\ \mathcal{L}\boldsymbol{\Pi}_{k,m+\mu}\end{array}\right],\label{eq:weight_ls_next}
\end{align}
respectively.  Further, let $\tau_{d,m+\mu}^{'}$  be the time it takes to solve \eqref{eq:efficient_solve} for $\mathbf{w}_{k,n,m+\mu}$.  Figure \ref{fig:Timing_Stuff} illustrates $\tau_{d,m+\mu}/\tau_{d,m}$ in the first row and $\tau_{d,m+\mu}^{'}/\tau_{d,m}$ in the second for various choices of $m$ and $\mu$ and for $d=1,2,3$.  These frames were generated by averaging the elapsed computation times of 1000 instances of solving each of these systems of equations for each choice of $m$, $\mu$ and $d$. Comparing the contour plots in the first and second rows demonstrates that in every case there is improvement in computational efficiency when solving \eqref{eq:efficient_solve} instead of the full system of linear equations and that these improvements become more apparent as $d$ increases. Note that all computations presented here were performed on a workstation with two Intel$^\circledR$ Xeon$^\circledR$ CPU E5-2697 v3 processors, each running at 2.60GHz, and 256 GB of memory running MATLAB R2022b.
\begin{figure}
\begin{center}
\includegraphics[width=\linewidth]{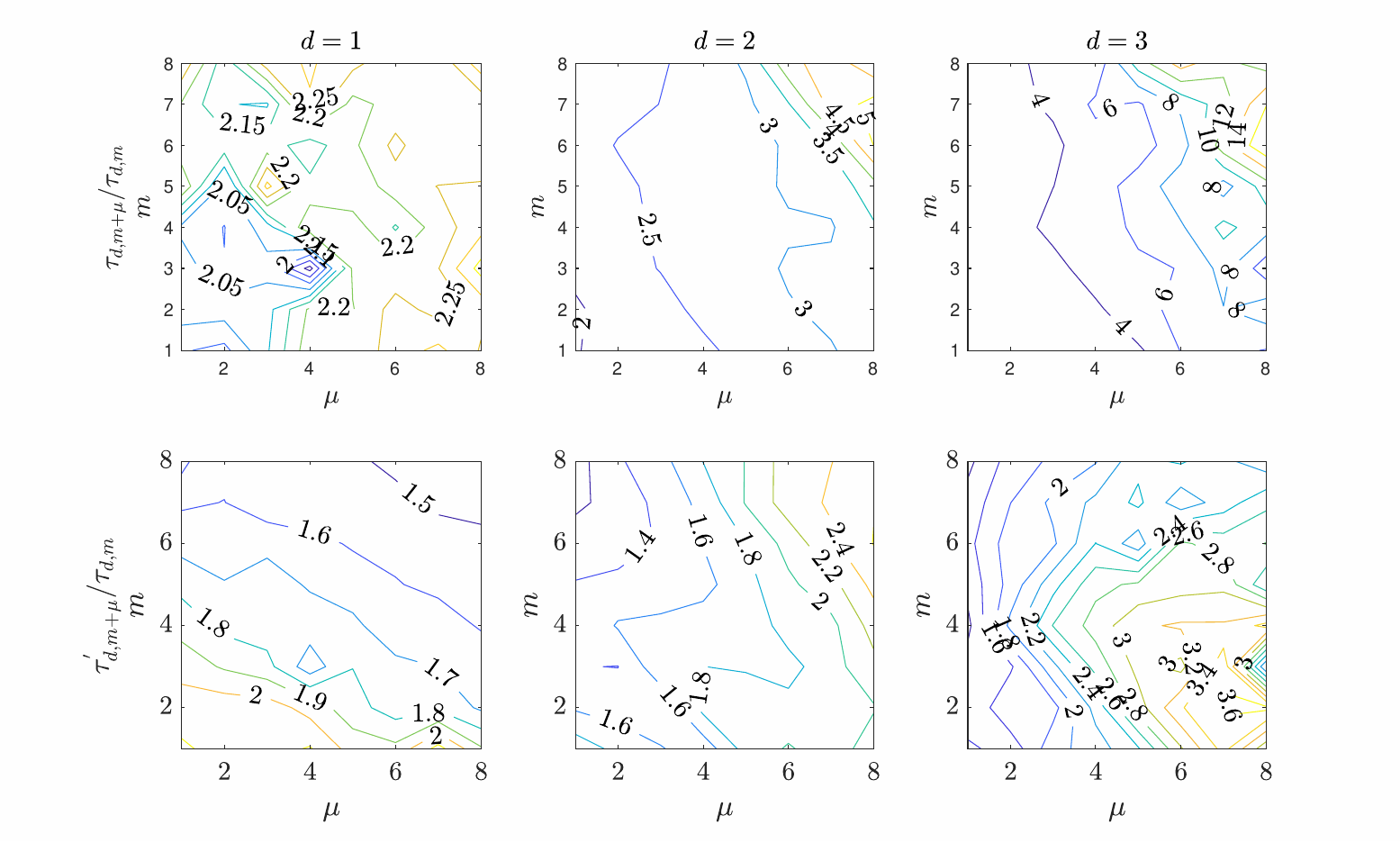}
\end{center}
\caption{Demonstrations of how many times longer it takes to solve \eqref{eq:weight_ls_next} (with average computation time $\tau_{d,m+\mu}$) and \eqref{eq:efficient_solve} (with average computation time $\tau_{d,m+\mu}^{'}$) relative to the time it takes to solve \eqref{eq:weight_ls} (with average computation time $\tau_{d,m}$).  These frames were generated by averaging the elapsed computation times of 1000 instances of solving each of these systems of equations for each choice of $m$, $\mu$ and $d$.}
\label{fig:Timing_Stuff}
\end{figure}

\subsection{A Naive Algorithm for Adaptive Kernel Based Approximation} \label{sec:Naive_Algorithm}

The focus of this paper is the development of an efficient error estimate for local ``$h$-adaptive" kernel based methods.  Such approaches to adaptation add new points to the set $\mathcal{S}$ in a neighborhood of those points in $\mathcal{X}_{0}$ where the error estimate is found to be too great.  This has the effect of locally decreasing the size of $h_{k,n}$.  With the estimate available, there are many possibilities for choosing the locations of these new points.  The algorithm can be generically described in the following steps:

    \begin{algorithmic}[1]
    \State Given a set of $N^{(0)}$ nodes, $\mathcal{S}^{(0)}$, a set of $K^{(0)}$ points, $\mathcal{X}_{0}^{(0)}=\{\mathbf{x}_{k,0}^{(0)}\}_{k=1}^{K^{(0)}}$, a desired tolerance, $\varepsilon$, and a maximum number of refinement levels, $l_{\mbox{max}}$. (And, if necessary, a tesselation $T^{(l)}$ of $S^{(l)}$.)
    \State Set $l=0$.
    \While{$l\leq l_{\mbox{max}}$}
        \State Set $\mathcal{R}^{(l)}=\emptyset$
        \For{$k\in\{1,2\ldots,K^{(l)}\}$}
        \State Determine the $n$ points in $S^{(l)}$ nearest to $\mathbf{x}_{k,0}^{(l)}$. Call this set of points $\mathcal{N}_{k,n}^{(l)}$.
        \If{$l=0$}
        \State Set $\mathcal{R}^{(l)}=\mathcal{R}^{(l)}\bigcup k$
        \Else
        \If{there is an index $i$ such that $\mathbf{x}_{k,0}^{(l)}=\mathbf{x}_{i,0}^{(l-1)}$}
        \If{$\mathcal{N}_{k,n}^{(l)}\setminus\mathcal{N}_{i,n}^{(l-1)}\neq\emptyset$}
        \State Set $\mathcal{R}^{(l)}=\mathcal{R}^{(l)}\bigcup k$
        \EndIf
        \Else
        \State Set $\mathcal{R}^{(l)}=\mathcal{R}^{(l)}\bigcup k$
        \EndIf
        \EndIf
        \EndFor
        \If{$\mathcal{R}^{(l)}=\emptyset$}
        \State Break and return $\mathcal{L}_{k}^{(l)}s_{k,n,m}$ as the local approximation of $\mathcal{L}_{k}^{(l)}f$ associated with each $\mathbf{x}_{k,0}^{(l)}\in\mathcal{X}_{0}^{(l)}$
        \EndIf
        \State Set $\mathcal{S}^{(l+1)}=\mathcal{S}^{(l)}$ and $\mathcal{X}_{0}^{(l+1)}=\mathcal{X}_{0}^{(l)}$.  (And, if necessary, set $T^{(l+1)}=T^{(l)}$.)
        \For{$k\in\mathcal{R}^{(l)}$}
        \State Compute $\mathcal{L}_{k}^{(l)}s_{k,n,m}$ and $\mathcal{L}_{k}^{(l)}s_{k,n,m}$ utilizing the nodes in $\mathcal{N}_{k,n}^{(l)}$.
        \If{$\lVert\mathcal{L}_{k}^{(l)}s_{k,n,m}-\mathcal{L}_{k}^{(l)}s_{k,n,m+\mu}\lVert_{2}>\varepsilon$}
        \State Add new nodes to the set $\mathcal{S}^{(l+1)}$ in a neighborhood of $\mathbf{x}_{k,0}^{(l)}$ and update the set $\mathcal{X}_{0}^{(l+1)}$ (and, if necessary, update $T^{(l+1)}$).
        \EndIf
        \EndFor
        \State Set $K^{(l+1)}=|\mathcal{X}_{0}^{(l+1)}|$, then increment $l$ by 1.
    \EndWhile
    \end{algorithmic}

Notice that in this implementation, at each level $l$, the approximation $\mathcal{L}_{k}^{(l)}s_{k,n,m}$ is updated for each $\mathbf{x}_{k,0}^{(l)}\in\mathcal{X}_{0}^{(l)}$ that has a set of nearest neighbors that has changed with the addition of new points to the set $\mathcal{S}^{(l)}$.  This is not necessary, and it is possible to update the approximation only for those points $\mathbf{x}_{k,0}^{(l)}$ for which the error estimate in line 22 violates the specified tolerance, reducing the computational cost.  Some computational experiments were performed with this method that reduces the cost; however, the best performance is achieved by updating the approximation for all points whose nearest neighbors have changed.

In all of the results presented here, the set $\mathcal{S}^{(0)}$ begins with $10^{d}$ equally spaced points in $[-1,1]^{d}$. Implementation of line 21 of the algorithm then depends upon the operation $\mathcal{L}$.  First, for approximating definite integrals at level $l$ of the algorithm note that
\begin{align}
    \int\limits_{\Omega}f(\mathbf{x})d\mathbf{x} = \sum\limits_{k=1}^{K^{(l)}}\int\limits_{\omega_{k}^{(l)}}f(\mathbf{x})d\mathbf{x},\nonumber
\end{align}
where $\omega_{k}^{(l)}$ is a portion of $\Omega$ associated with a simplex $t_{k}^{(l)}$, with midpoint $\mathbf{x}_{k,0}^{(l)}$, in a set of simplices constructed on the set $\mathcal{S}^{(l)}$.  A set of new nodes, that may contain the point $\mathbf{x}_{k,0}^{(l)}$, is added to $\mathcal{S}^{(l+1)}$. In this implementation $\mathbf{x}_{k,0}^{(l)}$, and the midpoints of the facets of the simplices (e.g., edges of triangles when $d=2$) are also added. Then $\mathbf{x}_{k,0}^{(l)}$ and $t_{k}^{(l)}$ are removed from $\mathcal{X}_{0}^{(l+1)}$ and $T^{(l+1)}$, respectively.  To the set $T^{(l+1)}$ are added the simplices in a tesselation of the set of points containing the vertices of $t_{k}^{(l)}$ and the points that have been added to $\mathcal{S}^{(l)}$.  The midpoints of these new simplices are added to the set $\mathcal{X}_{0}^{(l+1)}$.

On the other hand, for approximating a derivative when $d=1$, let $\mathbf{x}_{k,0,1}^{(l)}$ and $\mathbf{x}_{k,0,2}^{(l)}$ be the two points in $\mathcal{N}_{k,n}^{(l)}$ closest to $\mathbf{x}_{k,0}^{(l)}$. The two points $(\mathbf{x}_{k,0}^{(l)}+\mathbf{x}_{k,0,1}^{(l)})/2$ and $(\mathbf{x}_{k,0}^{(l)}+\mathbf{x}_{k,0,2}^{(l)})/2$ are added to both $\mathcal{S}^{(l+1)}$ and $\mathcal{X}_{0}^{(l+1)}$, while nothing is removed from either set.  When $d=2$, let $h^{(0)}=2/9$ (the initial horizontal/vertical spacing between adjacent nodes).  At level $l$ the points $\mathbf{x}_{k,0}^{(l)}\pm h^{(0)}/2^{l+1}\left[\begin{array}{cc}1 & 0\end{array}\right]^{T}$, $\mathbf{x}_{k,0}^{(l)}\pm h^{(0)}/2^{l+1}\left[\begin{array}{cc}0 & 1\end{array}\right]^{T}$, $\mathbf{x}_{k,0}^{(l)}\pm (\sqrt{2}/2)h^{(0)}/2^{l+1}\left[\begin{array}{cc}1 & 1\end{array}\right]^{T}$,$\mathbf{x}_{k,0}^{(l)}\pm (\sqrt{2}/2)h^{(0)}/2^{l+1}\left[\begin{array}{cc}1 & -1\end{array}\right]^{T}$, $\mathbf{x}_{k,0}^{(l)}\pm (\sqrt{2}/2)h^{(0)}/2^{l+1}\left[\begin{array}{cc}-1 & -1\end{array}\right]^{T}$, and $\mathbf{x}_{k,0}^{(l)}\pm (\sqrt{2}/2)h^{(0)}/2^{l+1}\left[\begin{array}{cc}-1 & 1\end{array}\right]^{T}$ are added to the both $\mathcal{S}^{(l+1)}$ and $\mathcal{X}_{0}^{(l+1)}$, while nothing is removed from either set.

These strategies for adding nodes are certainly not the only ones available and may not be optimal; however, they perform well enough to illustrate the performance of even naively implemented algorithms based on the error estimate.

\section{Computational Experiments} \label{sec:Experiments}

Demonstrations of the agreement between the error estimate and actual error are provided in the following section when utilizing the algorithm described in section \ref{sec:Naive_Algorithm}.  These demonstrations explore two common linear operations in both $d=1$ and $d=2$ for two test functions with localized features that require significant refinement to be resolved.

\subsection{Test Functions and Kernel Selection for Computational Experiments}

Computational experiments were performed for $d=1$ and $d=2$ using the test functions
\begin{align}
    f_{1}(\mathbf{x}) = \sum\limits_{i=1}^{2d}\frac{1}{1+a\left\lVert\mathbf{x}-\mathbf{y}_{i}\right\rVert_{2}^{2}}\nonumber
\end{align}
and
\begin{align}
    f_{2}(\mathbf{x}) = \sum\limits_{i=1}^{2d}\mbox{e}^{-a\left\lVert\mathbf{x}-\mathbf{y}_{i}\right\rVert_{2}^{2}},\nonumber
\end{align}
both with $\mathbf{y}_{i}$ a randomly chosen shift in $(-1,1)^{d}$. Graphically these functions are similar in character with features that become more localized as $a$ increases.  However, the series expressions for the terms in $f_{1}$ and $f_{2}$ behave quite differently.  That is,
\begin{align}
\frac{1}{1+a\left\lVert\mathbf{x}-\mathbf{y}\right\rVert_{2}^{2}}=& \sum\limits_{j=0}^{\infty}(-1)^{j}a^{j}\left(\left\lVert\mathbf{x}-\mathbf{y}\right\rVert_{2}^{2}\right)^{j}=\sum\limits_{j=0}^{\infty}\sum\limits_{|\boldsymbol\beta|=j}\frac{j!(-1)^{j}a^{j}}{\boldsymbol\beta!}\left(\mathbf{x}-\mathbf{y}\right)^{2\boldsymbol\beta}\nonumber
\end{align}
has radius of convergence $\left\lVert\mathbf{x}-\mathbf{y}\right\rVert_{2}<1/\sqrt{a}$ with terms that grow for all $j$ when outside the radius of convergence, and
\begin{align}
    \mbox{e}^{-a\left\lVert\mathbf{x}-\mathbf{y}\right\rVert_{2}^{2}} =& \sum\limits_{j=0}^{\infty}(-1)^{j}\frac{a^{j}}{j!}\left(\left\lVert\mathbf{x}-\mathbf{y}\right\rVert_{2}^{2}\right)^{j}= \sum\limits_{j=0}^{\infty}\sum\limits_{|\boldsymbol\beta|=j}\frac{(-1)^{j}a^{j}}{\boldsymbol\beta!}\left(\mathbf{x}-\mathbf{y}\right)^{2\boldsymbol\beta}\nonumber
\end{align}
has infinite radius of convergence and terms that grow only until finite $j$ (for each fixed $\left\lVert\mathbf{x}-\mathbf{y}\right\rVert_{2}$).

While there are many options to choose from for the kernel used in the local approximations, the results here considered the use of $\varphi(r)=r^3$.

\subsection{Results in 1-Dimension}

To illustrate the performance of the algorithm described in section \ref{sec:Naive_Algorithm} for $d=1$ it was applied to approximate the action of both
\begin{align}
    \mathcal{L}f = \int\limits_{-1}^{1}f(x)dx\mbox{ }\bigg(\mbox{i.e., }\mathcal{L}_{k}f=\int\limits_{\omega_{k}}f(x)dx\bigg)\nonumber
\end{align}
and
\begin{align}
    \mathcal{L}f = \frac{d}{dx}f(x)\mbox{, }x\in[-1,1]\mbox{ }\bigg(\mbox{i.e., }\mathcal{L}_{k}f=\frac{d}{dx}f(x)\bigg\lvert_{x=x_{k,0}}\bigg).\nonumber
\end{align}
Consider the choice of $a=1000$, $N^{(0)}=10$, $m=1$, and $\mu=2$.  Figure \ref{fig:Single_Adaptive_Quadrature_One_D_Example} illustrates a single example of adaptive quadrature for evaluating the integral of $f_{2}(x)$ over the interval $[-1,1]$.  The top frame illustrates $f_{2}$ and the locations of the nodes required to achieve a tolerance of $\varepsilon=10^{-5}$, while the middle frame shows the node spacing, corresponding to the widths of the subintervals being integrated over.  Both the top and middle frames indicate that the algorithm is adaptively placing points with increased density where the function is changing rapidly (i.e., the derivative is larger).  This is the expected behavior.  The bottom frame then illustrates that the absolute error estimate and actual absolute error are in agreement and both meet the prescribed tolerance with $N=93$ nodes on $[-1,1]$.

\begin{figure}
\begin{center}
\includegraphics[width=\linewidth]{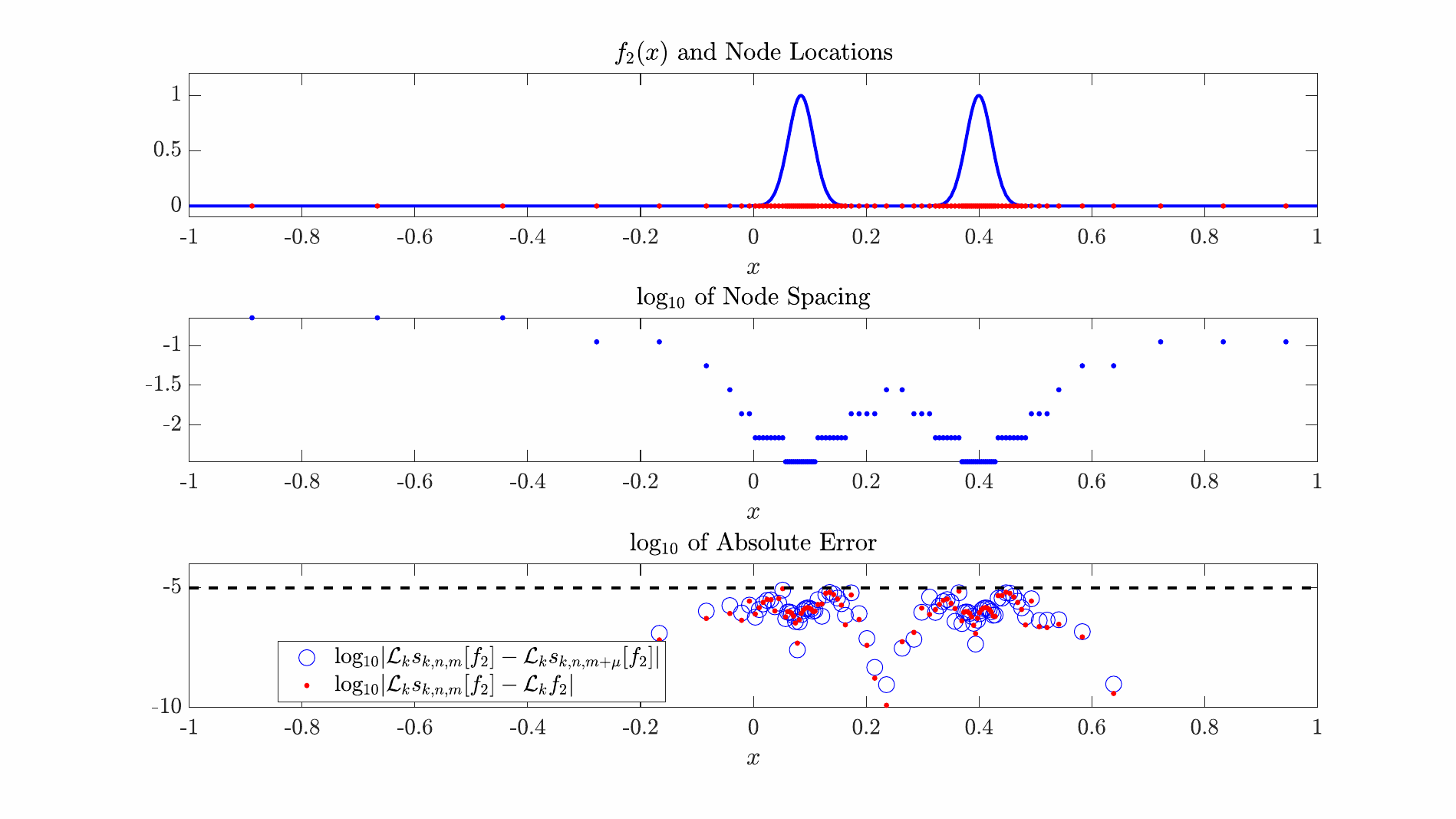}
\end{center}
\caption{An example of adaptive approximate numerical quadrature of $f_{2}$ for $x\in[-1,1]$ with $a=1000$, $m=1$, $\mu=2$, $\varepsilon=10^{-5}$, $y_{1}=0.084435845510910$, $y_{2}=0.399782649098896$, and $N=93$. The top frame indicates the function and node locations.  The middle frame shows the node spacing, which corresponds to the widths of the intervals being integrated over.  The bottom frame illustrates agreement of the absolute error estimate and actual absolute error and satisfaction of the prescribed tolerance. Node locations where no marker for the error estimate or actual error appears indicate estimates and actual errors that fall well below the vertical limits of the plot. }
\label{fig:Single_Adaptive_Quadrature_One_D_Example}
\end{figure}

Similarly, figure \ref{fig:Single_Adaptive_Differentiation_One_D_Example} shows a single example of adaptive differentiation of $f_{2}(x)$ on the interval $[-1,1]$.  Since the error estimate is utilized to approximate the \textit{absolute} error in the derivative, more points are required to achieve a specific tolerance in comparison to adaptive quadrature.  This is because the derivatives of both $f_{1}(x)$ and $f_{2}(x)$ are an order of magnitude larger than the maximum value of their integrals over $[-1,1]$.  In fact, the maximum of the derivative for any choice of $a$ can be as large as $2(3\sqrt{3a}/8)$ for $f_{1}$ and $2(\sqrt{2a\mbox{e}^{-1}})$ for $f_{2}$.  For this reason, in this example the algorithm is applied to computing the derivative of $f_{2}(x)$ with $\varepsilon = 10^{-2}$ so that individual nodes are visually distinguishable except where the nodes are adaptively placed with the greatest density.  Still, the top and middle frames indicate that the algorithm is again adaptively placing points with increased density where the derivative is large and the bottom frame illustrates that the error estimate and actual error are in agreement and meet the prescribed tolerance, now with $N=3093$ nodes on $[-1,1]$.

\begin{figure}
\begin{center}
\includegraphics[width=\linewidth]{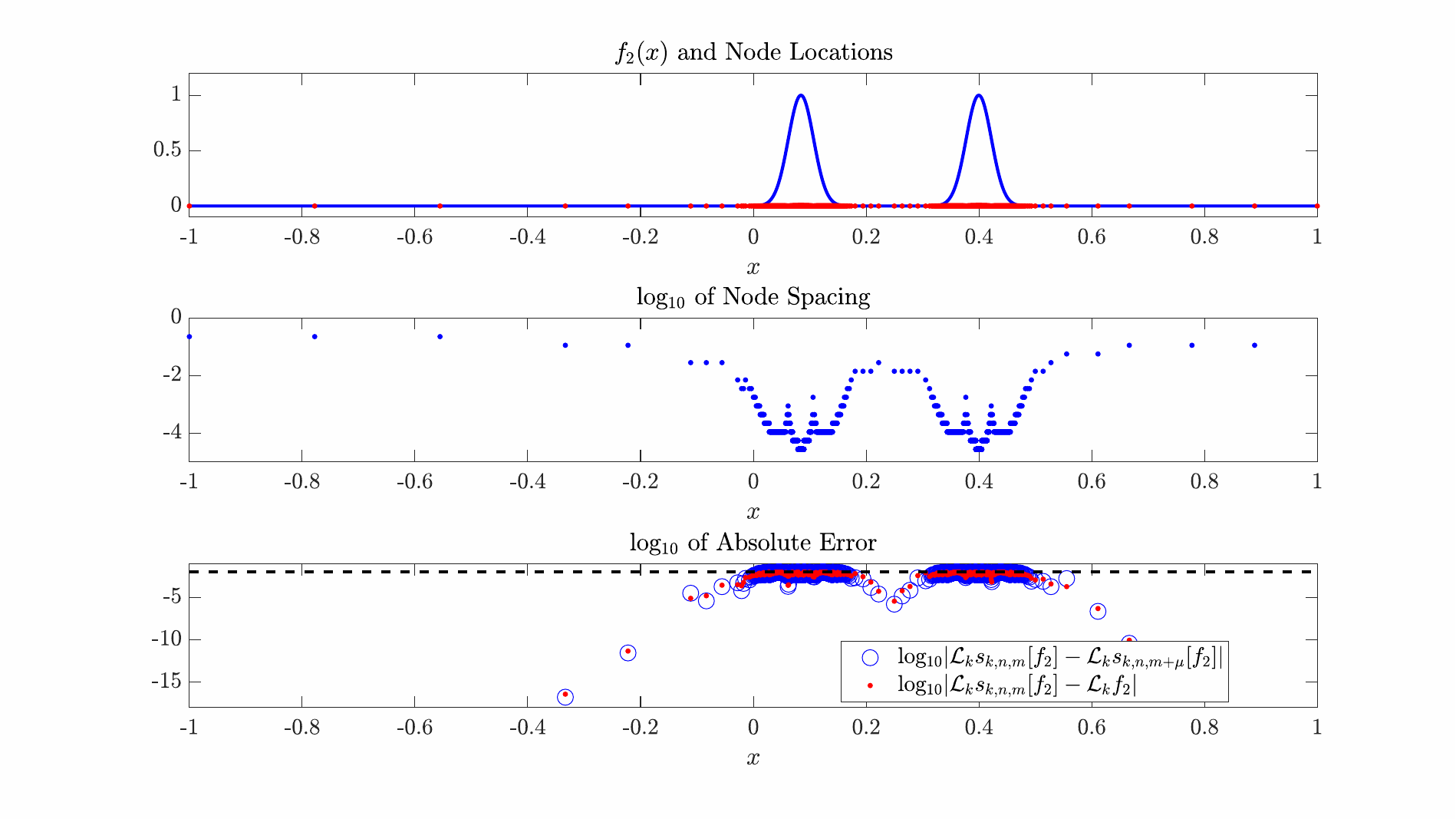}
\end{center}
\caption{An example of approximation differentiation of $f_{2}$ for $x\in[-1,1]$ with $a=1000$, $m=1$, $\mu=2$, $\varepsilon=10^{-5}$, $y_{1}=0.084435845510910$, $y_{2}=0.399782649098896$, and $N=3093$. The top frame indicates the function and node locations.  The middle frame shows the node spacing.  The bottom frame illustrates agreement of the absolute error estimate and actual absolute error and satisfaction of the prescribed tolerance. Node locations where no marker for the error estimate or actual error appears indicate estimates and actual errors that fall well below the vertical limits of the plot. }
\label{fig:Single_Adaptive_Differentiation_One_D_Example}
\end{figure}

The choice of $\mu=2$ in these examples is motivated by an important observation and computational experiments.  First, for functions that exhibit even or odd symmetry about one of the points $x_{k,0}\in\mathcal{X}_{0}$ it is possible that $s_{k,n,m}[f](x)$ and $s_{k,n,m+1}[f](x)$ will account for exactly the same terms in the Taylor formula so that their difference may be small even when the actual error is large.  Second, figure \ref{fig:Muu_Comparision_Adaptive_Quadrature_One_D} demonstrates the number of nodes, $N$, required to locally reach various prescribed tolerances, $\varepsilon\in[10^{-7},10^{-4}]$, for $a=100$ and both $f=f_{1}$ (solid curves) and $f=f_{2}$ (dashed curves) for each value of $m=1,2,3,4$ and $\mu=2,3$.  There is no noticeable improvement, which would be indicated by fewer nodes being required, when increasing $\mu$ from $2$ to $3$, so that the discussion in section \ref{sec:Computational_Cost} indicates that $\mu=2$ should be preferred for the reduced computational expense.  Further computational experiments were performed with $\mu=1$ and $\mu>3$, all of which indicated that $\mu=2$ is a reasonable choice.

Figure \ref{fig:Muu_Comparision_Adaptive_Quadrature_One_D} also provides a comparison between the adaptive algorithm here applied to quadrature and the more familiar adaptive trapezoidal rule.  The implementation of the adaptive trapezoidal rule mirrors the presentation in \cite[Section 5.5]{atkinson}.  The figure illustrates that even when the adaptive kernel method is of the same anticipated order as the adaptive trapezoidal rule (i.e., when $m=1$) the method here requires fewer nodes to achieve the desired tolerance.
\begin{figure}
\begin{center}
\includegraphics[width=0.75\linewidth]{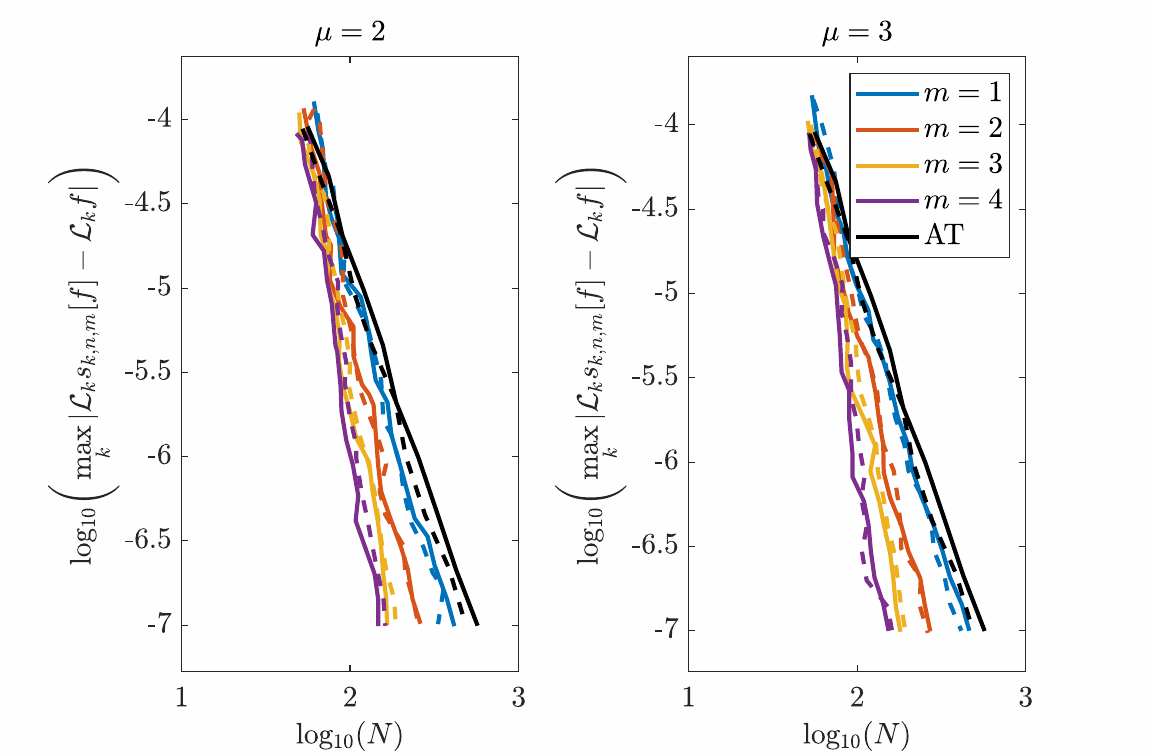}
\end{center}
\caption{Log base 10 of the average maximum absolute error versus log base 10 of the average number of nodes required to achieve that error for approximate definite integration of $f_{1}$ (solid curves) and $f_{2}$ (dashed curves) for $x\in[-1,1]$ with $a=100$.  The average is taken over 10 choices of $y_{1}$ and $y_{2}$ for each value of $\varepsilon$, corresponding closely to the error shown.  The left frame corresponds to a choice of $\mu=2$ while in the right frame $\mu=3$.  In both cases curves are shown for $m=1,2,3,4$ along with analogous results for an adaptive trapezoidal rule (AT in the legend) that utilizes the difference between Simpson's rule and trapezoidal rule as an error estimate.}
\label{fig:Muu_Comparision_Adaptive_Quadrature_One_D}
\end{figure}

The impact of the radius of convergence of and the growth of the terms in the Taylor formula on the required number of nodes for each test function was also assessed for both numerical quadrature and differentiation.  Both the radius of convergence and growth of the terms were controlled by varying $a$ (here $a\in[1,1000]$) and the number of required nodes to achieve specified tolerances ($\varepsilon\in[10^{-7},10^{-4}]$ for quadrature and $\varepsilon\in[10^{-6},10^{-3}]$ for differentiation) was recorded. The contours in each frame of figures \ref{fig:Adaptive_Quadrature_One_D_Muu_2} and \ref{fig:Adaptive_Differentiation_One_D_Muu_2} illustrate log base 10 of the number of nodes $N$ required for each choice of $a$ and $\varepsilon$, with each frame corresponding to a different value of $m=1,2,3,4$. Due to the memory requirements for storing all diagnostic data used in the analyses presented here, the algorithm was forced to terminate when $N>10^{5.5}$. This limit should not be necessary for more efficient implementations of the algorithm and here only impacted the results for numerical differentiation in the case of $m=1$.  The consequences of this limit are illustrated by the missing curves down and to the right of the curve indicating $N=10^{5}$ in the top left frame of figure \ref{fig:Adaptive_Differentiation_One_D_Muu_2}.  As expected, an increase in $a$, which leads to more localized features in both $f_{1}$ and $f_{2}$ and faster growth of their derivatives, necessitates an increase in $N$ to achieve the same tolerance $\varepsilon$.  Consistent with remarks \ref{rem:Approx_Order} and \ref{rem:estimate1} and figure \ref{fig:Muu_Comparision_Adaptive_Quadrature_One_D} the number of nodes required for each value of $a$ and $\varepsilon$ decreases with increasing $m$.  However, despite the finite and infinite radii of convergence of the Taylor formulas of $f_{1}$ and $f_{2}$, respectively, the algorithm performs similarly, demonstrating that it gracefully handles the growth in the terms of the series and finite radii of convergence.

\begin{figure}
\begin{center}
\includegraphics[width=0.75\linewidth]{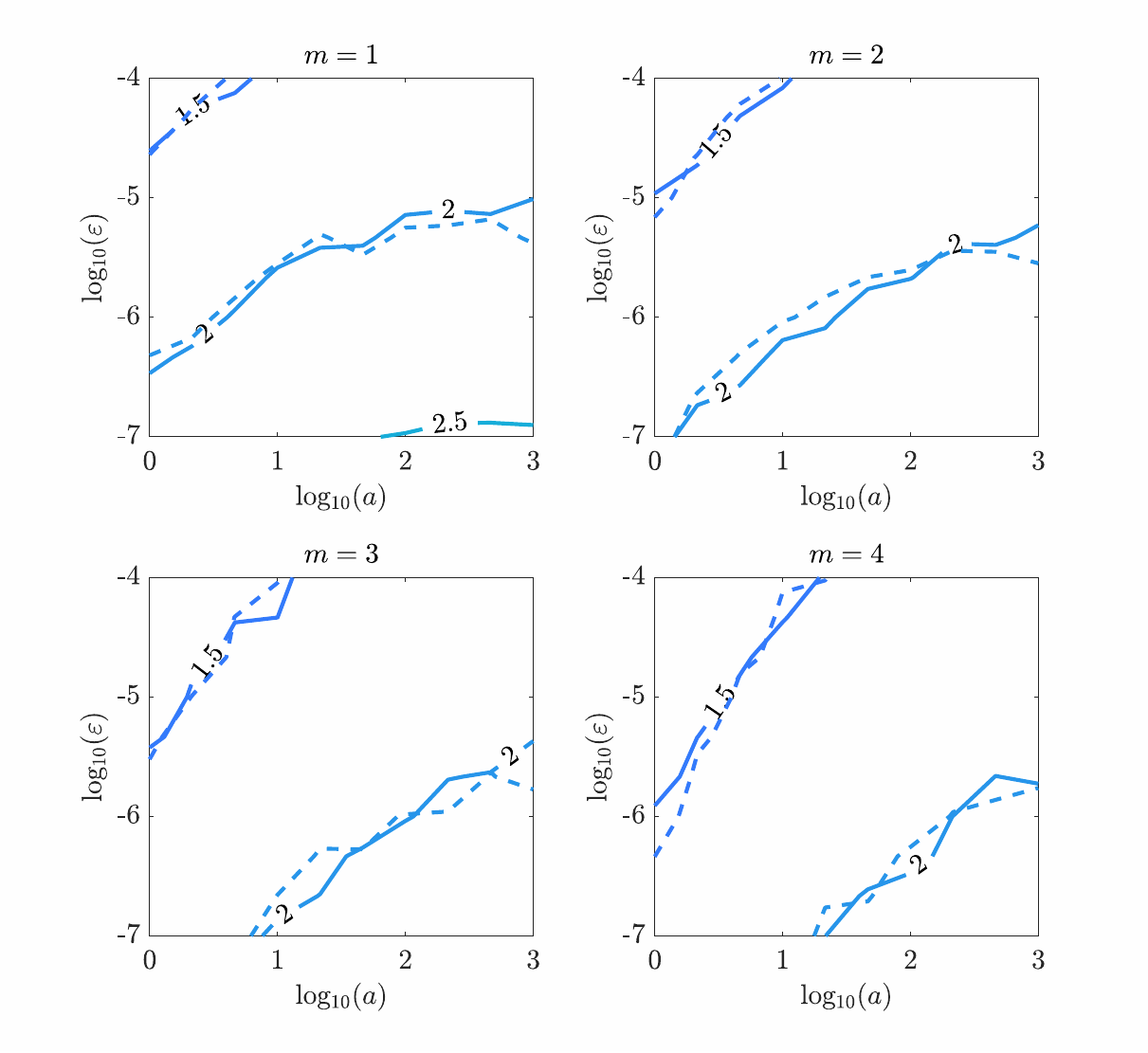}
\end{center}
\caption{An illustration of the number of nodes required to achieve an estimated error tolerance of $\varepsilon$ for approximate definite integration of $f_{1}$ (solid curves) and $f_{2}$ (dashed curves) for $x\in[-1,1]$ with $a\in[1,1000]$.  The contours indicate the log base 10 of the mean of the number of required nodes for ten random choices of $y_{1}$ and $y_{2}$ for each value of $a$ and $\varepsilon$.  In all cases $\mu=2$.}
\label{fig:Adaptive_Quadrature_One_D_Muu_2}
\end{figure}

\begin{figure}
\begin{center}
\includegraphics[width=0.75\linewidth]{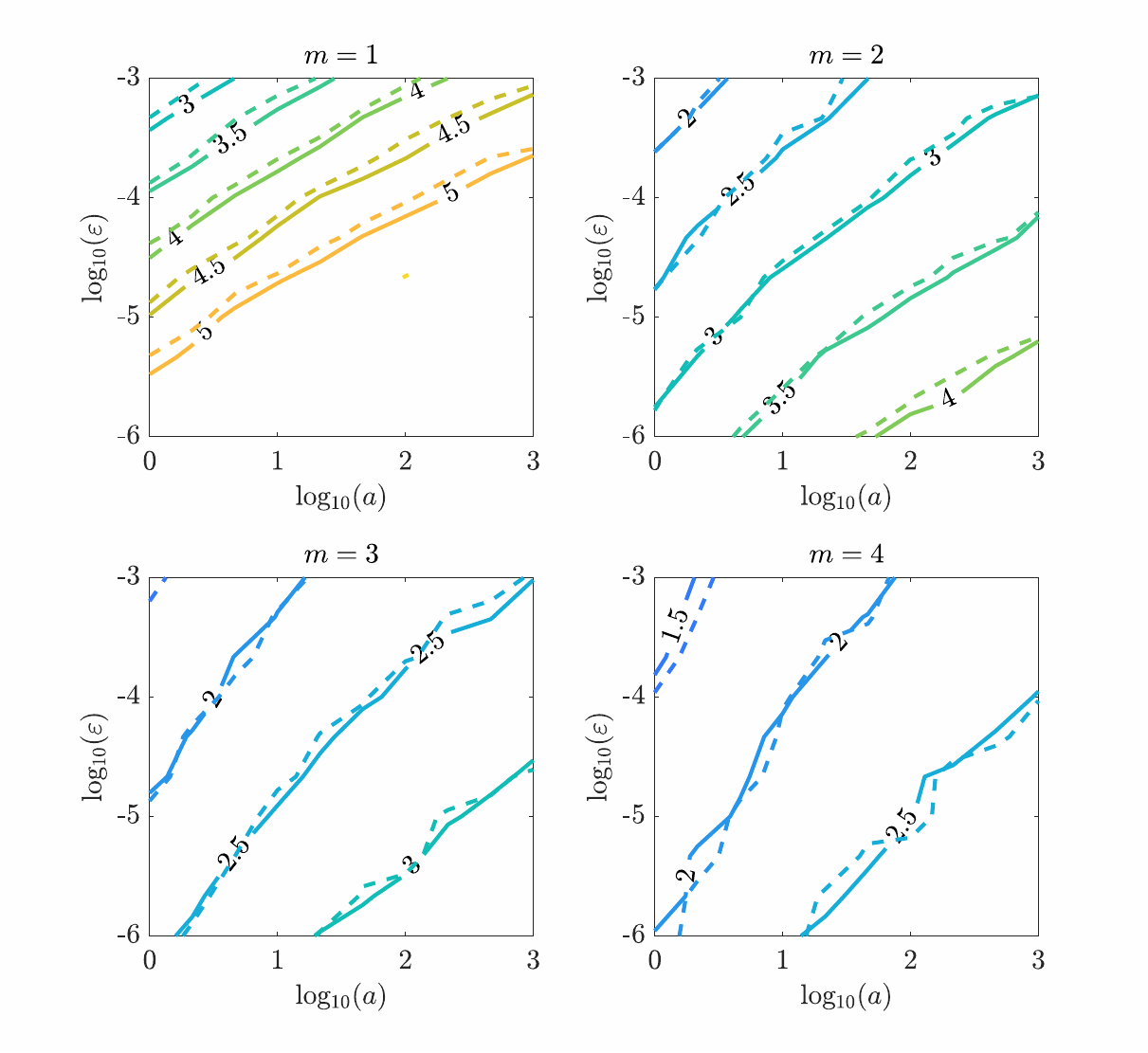}
\end{center}
\caption{An illustration of the number of nodes required to achieve an estimated error tolerance of $\varepsilon$ for approximate differentiation of $f_{1}$ (solid curves) and $f_{2}$ (dashed curves) for $x\in[-1,1]$ with $a\in[1,1000]$.  The contours indicate the log base 10 of the mean of the number of required nodes for ten random choices of $y_{1}$ and $y_{2}$ for each value of $a$ and $\varepsilon$.  In all cases $\mu=2$.}
\label{fig:Adaptive_Differentiation_One_D_Muu_2}
\end{figure}

\subsection{Examples in 2-Dimensions}

For $d=2$ computational experiments were performed for both
\begin{align}
    \mathcal{L}f=\int\limits_{\Omega}f(\mathbf{x})d\mathbf{x},\mbox{ }\Omega = [-1,1]^{2}\mbox{ }\bigg(\mbox{i.e., }\mathcal{L}_{k}f=\int\limits_{\omega_{k}}f(\mathbf{x})d\mathbf{x}\bigg)\nonumber
\end{align}
and
\begin{align}
    \mathcal{L}f = \nabla f,\mbox{ }\mathbf{x}\in[-1,1]^{2}\mbox{ }\bigg(\mbox{i.e., }\mathcal{L}_{k}f=\nabla f(\mathbf{x})\bigg\lvert_{\mathbf{x}=\mathbf{x}_{k,0}}\bigg).\nonumber
\end{align}

Consider the choice of $a=1000$, $N^{(0)}=100$, $m=4$, and $\mu=2$.  Figure \ref{fig:Single_Adaptive_Quadrature_Two_D_Example} illustrates a single example of adaptive quadrature for evaluating the integral of $f_{2}(\mathbf{x})$ over the domain $[-1,1]^{2}$ with $\varepsilon = 10^{-6}$.  The top left frame illustrates the magnitude of $f_{2}$, while the top right frame shows the node spacing, represented by the area of the triangles $\omega_{k}$.  The top right frame again indicates that the algorithm is adaptively placing points with increased density where the function is changing rapidly.  The bottom frames illustrate that the absolute error estimate and actual absolute error are in agreement with $N=2366$.

\begin{figure}
\begin{center}
\includegraphics[width=\linewidth]{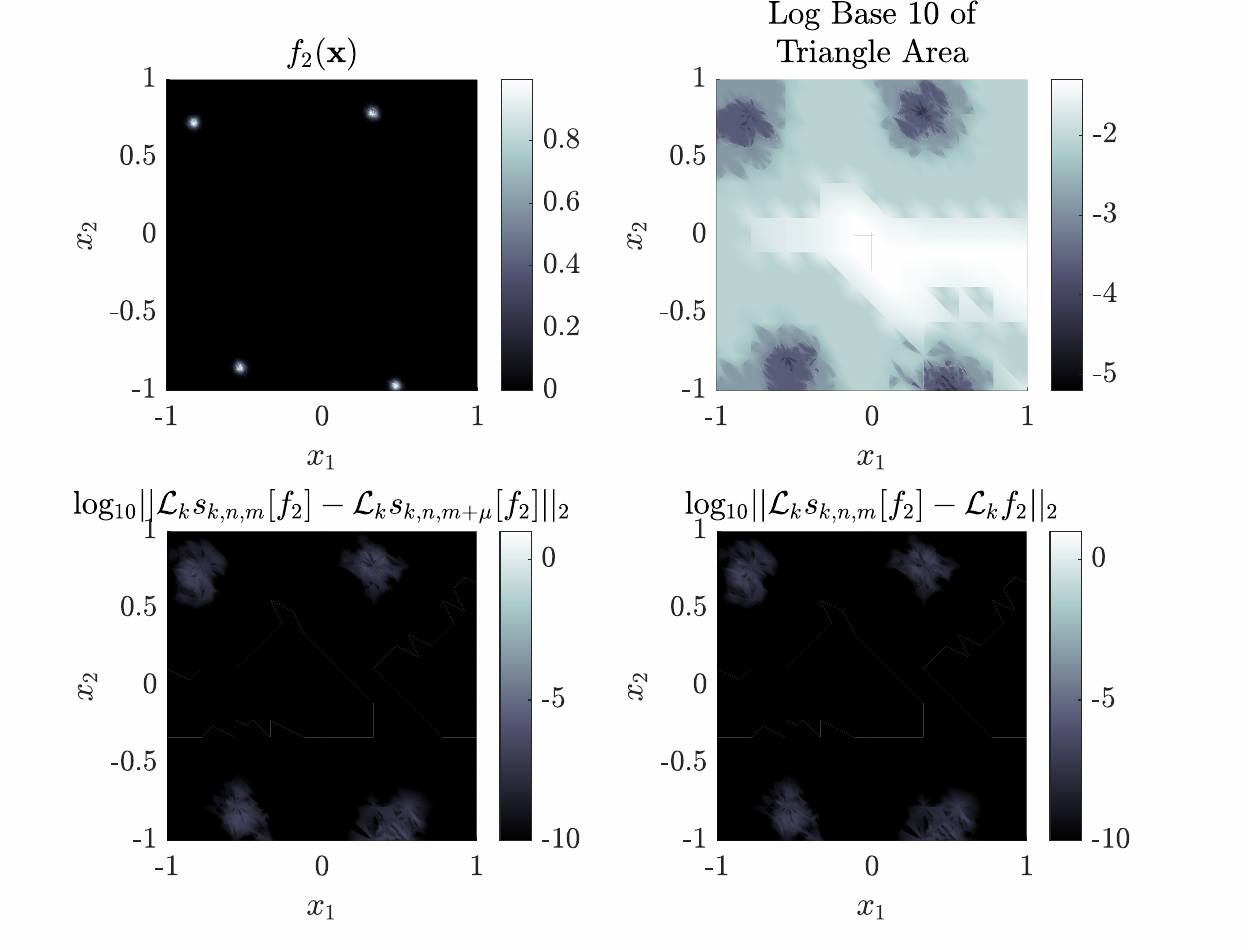}
\end{center}
\caption{An example of adaptive approximate numerical quadrature of $f=f_{2}$ for $x\in[-1,1]$ with $a=1000$, $m=4$, $\mu=2$, $n=M_{d,m+\mu}=28$, $\varepsilon=10^{-6}$, $\mathbf{y}_{1}=\left[\begin{array}{cc}0.322471807186779  & 0.784739294760742\end{array}\right]^{T}$, $\mathbf{y}_{2}=\left[\begin{array}{cc}0.471357153710612 & -0.964237266730882\end{array}\right]^{T}$, $\mathbf{y}_{3}=\left[\begin{array}{cc}-0.824125584316469 &  0.721758033391102\end{array}\right]^{T}$, $\mathbf{y}_{4}=\left[\begin{array}{cc}-0.526514007034680 & -0.847278799561768\end{array}\right]^{T}$ and $N=2366$. The top left frame illustrates the function the magnitude of the function.  The top right frame shows the node spacing measured by the areas of the triangles $\omega_{k}$.  The bottom frames illustrate agreement of the absolute error estimate and actual absolute error and satisfaction of the prescribed tolerance.}
\label{fig:Single_Adaptive_Quadrature_Two_D_Example}
\end{figure}

Similarly, figure \ref{fig:Single_Adaptive_Differentiation_Two_D_Example} shows a single example of adaptive differentiation of $f_{2}(\mathbf{x})$ on the domain $[-1,1]^{2}$.  Since the derivative can again be large, in this example the algorithm is applied to computing the derivative of $f_{2}$ with $\varepsilon = 10^{-2}$.  Still, the top right frame indicates that the algorithm is again adaptively placing points with increased density where the derivative is large and the bottom frames illustrate that the error estimate and actual error are in agreement and meet the prescribed tolerance, now with $N=14852$ nodes on $[-1,1]^{2}$.

\begin{figure}
\begin{center}
\includegraphics[width=\linewidth]{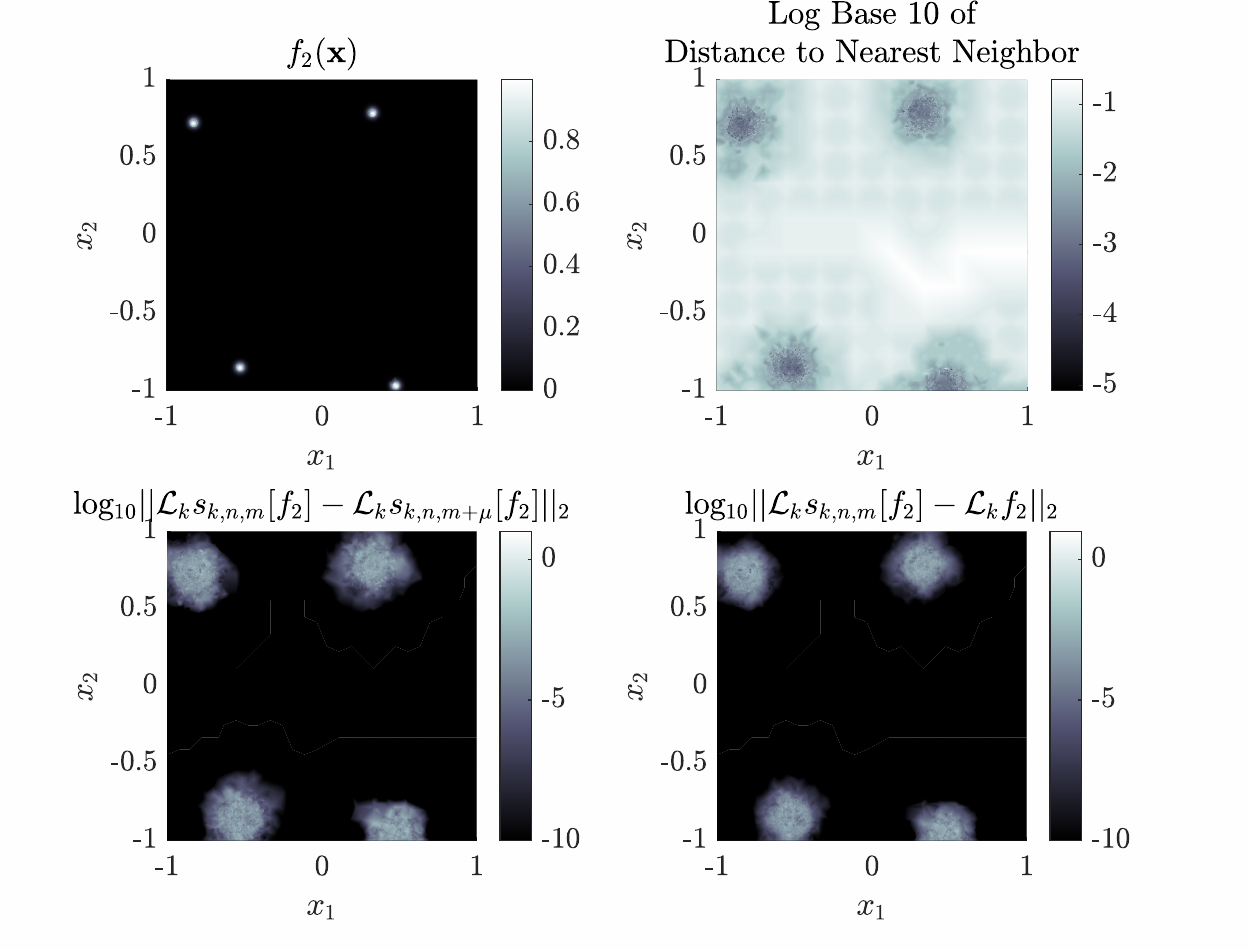}
\end{center}
\caption{An example of approximation differentiation of $f=f_{2}$ for $\mathbf{x}\in[-1,1]^{2}$ with $a=1000$, $m=4$, $\mu=2$, $n=M_{d,m+\mu}=28$ $\varepsilon=10^{-5}$, $\mathbf{y}_{1}=\left[\begin{array}{cc}0.322471807186779  & 0.784739294760742\end{array}\right]^{T}$, $\mathbf{y}_{2}=\left[\begin{array}{cc}0.471357153710612 & -0.964237266730882\end{array}\right]^{T}$, $\mathbf{y}_{3}=\left[\begin{array}{cc}-0.824125584316469 &  0.721758033391102\end{array}\right]^{T}$, $\mathbf{y}_{4}=\left[\begin{array}{cc}-0.526514007034680 & -0.847278799561768\end{array}\right]^{T}$, and $N=14852$. The top left frame illustrates the function the magnitude of the function.  The top right frame shows the node spacing measured by the distance of each node to its nearest neighbor.  The bottom frames illustrate agreement of the absolute error estimate and actual absolute error and satisfaction of the prescribed tolerance.}
\label{fig:Single_Adaptive_Differentiation_Two_D_Example}
\end{figure}

Again, the choice of $\mu=2$ in these examples is motivated by computational experiments.  Figure \ref{fig:Muu_Comparision_Adaptive_Quadrature_Two_D} is analogous to the presentation in figure \ref{fig:Muu_Comparision_Adaptive_Quadrature_One_D}, only now for approximate definite integration over $[-1,1]^{2}$.  There is again no noticeable improvement when increasing $\mu$ from $2$ to $3$, and $\mu=2$ is likewise preferred for the reduced computational expense.  However, it is clear in these examples that the algorithm performs better for a function whose Taylor series has infinite radius of convergence.
\begin{figure}
\begin{center}
\includegraphics[width=0.75\linewidth]{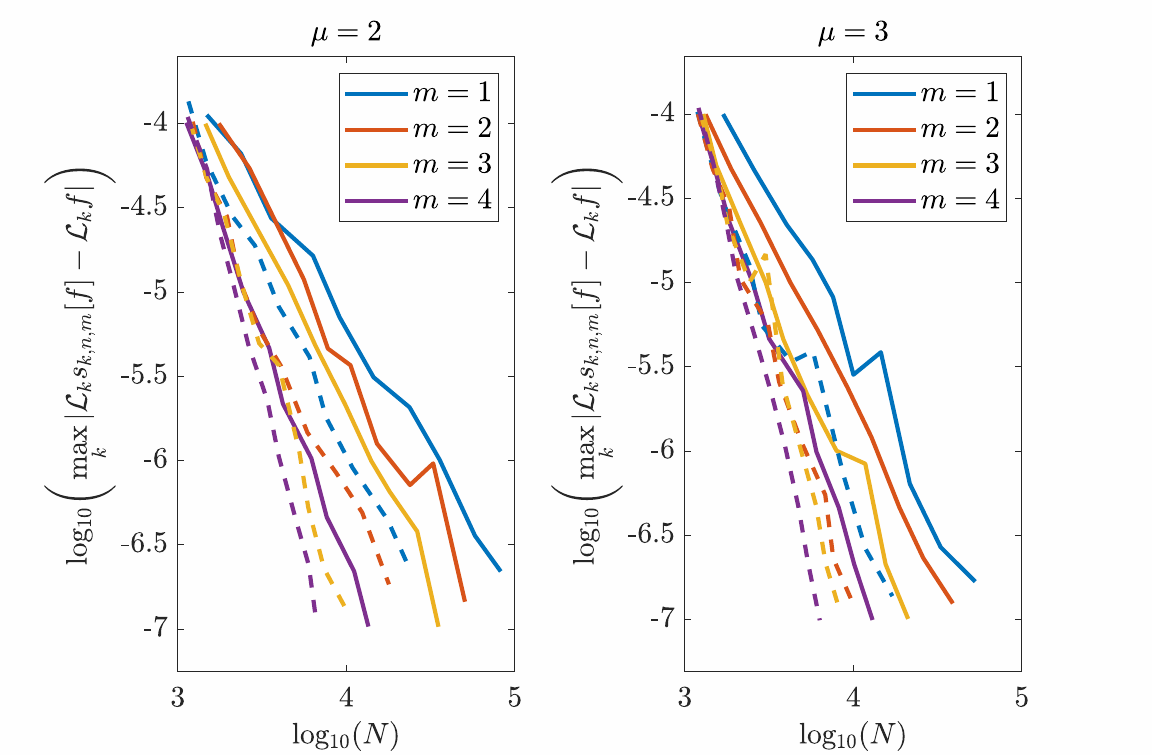}
\end{center}
\caption{Log base 10 of the average maximum absolute error versus log base 10 of the average number of nodes required to achieve that error for approximate definite integration of $f_{1}$ (solid curves) and $f_{2}$ (dashed curves) for $\mathbf{x}\in[-1,1]^{2}$ with $a=100$.  The average is taken over 10 choices of $\mathbf{y}_{1}$, $\mathbf{y}_{2}$, $\mathbf{y}_{3}$ and $\mathbf{y}_{4}$ for each value of $\varepsilon$.  Here $\varepsilon$ corresponds closely to the error shown.  The left frame corresponds to a choice of $\mu=2$ while in the right frame $\mu=3$.  In both cases curves are shown for $m=1,2,3,4$.}
\label{fig:Muu_Comparision_Adaptive_Quadrature_Two_D}
\end{figure}

\section{Conclusions} \label{sec:Conclusions}

This article presented a novel approach to constructing approximations to the action of linear operators that locally adapt the spacing of the discrete point set to achieve a prescribed error tolerance.  The approach described here utilized kernel methods similar to RBF-FD to overcome the difficulties experienced with an analogous use of polynomial interpolation in the presence of scattered nodes, particularly when $d>1$.   Computational experiments have shown that the estimate and actual absolute error are in agreement and that the estimate can be successfully used to indicate where refinement of the discrete node set is required.  

This study is not exhaustive.  In particular, there is significant opportunity to explore node placement strategies where the error estimate indicates the need for refinement.  Further, the choice to refine at every point where the estimate is computed that has a change in its nearest neighbors from one level to the next is likely unnecessary.  Therefore, approaches to refining only where the error indicator is larger than the prescribed tolerance are an important avenue for investigation.  

%
%
\section*{Funding}

This research was supported in part by the Joint Directed Energy Transition Office Atmospheric Physics Modeling and Simulation Technical Area Working Group.
%
%
%
%
%
%

\bibliographystyle{unsrt}      
\bibliography{references}   

\end{document}